\documentclass[12pt]{article}
\usepackage{changes}
\usepackage{color}        
\usepackage{hyperref}
\usepackage[psamsfonts]{amsfonts}
\usepackage{amsmath,amssymb,amsthm,bbm,latexsym}
\usepackage{graphicx}
\newtheorem{THM}{Theorem}[section]
\newtheorem{COR}[THM]{Corollary}
\newtheorem{EXA}[THM]{Example}
\newtheorem{REM}[THM]{Remark}
\newtheorem{LEM}[THM]{Lemma}
\newtheorem{PRP}[THM]{Proposition}
\newtheorem{DEF}[THM]{Definition}
\newtheorem{CON}[THM]{Conjecture}

\numberwithin{equation}{section}

\newcommand{\blank}[1]{}
\newcommand{\Maaak}[1]{{\color{black}{#1}}}

\newenvironment{proofsect}[1] 
{\vskip0.1cm\noindent{\bf #1.}\hskip0.5cm} 
 
\newcommand{\halmos}{\rule{1ex}{1.4ex}}
\newcommand{\eqa}{\begin{eqnarray}}
\newcommand{\ena}{\end{eqnarray}}
\newcommand{\eq}{\begin{equation}}
\newcommand{\en}{\end{equation}}
\newcommand{\eqs}{\begin{eqnarray*}}
\newcommand{\ens}{\end{eqnarray*}}
\def\qed{$\halmos$}

\def\emptyset{\varnothing} 
\def\ti{\to\infty} 

\def\n{\nu} 
\def\r{\varrho} 

\newcommand{\Z}     {\mathbb{Z}} 
\newcommand{\N}     {\mathbb{N}} 
\renewcommand{\P}   {\mathbb{P}} 
 
\newcommand{\E}     {\mathbb{E}}
 \newcommand{\T}     {\mathbb{T}}
 
 \newcommand{\floor}[1]{\left\lfloor #1 \right\rfloor}

\newcommand{\ssup}[1] {{{\scriptscriptstyle{({#1}})}}} 
\def\comment#1{} 
 \newcommand{\ra}{\rightarrow}

\newcommand{\sss}{\scriptscriptstyle} 

\newcommand{\Ccal}   {{\mathcal C }}

\newcommand{\Gcal}   {{\mathcal G }}

\newcommand{\Tcal}   {{\mathcal T }} 
 
\newcommand{\mc}[1]{\mathcal{#1}}

\newcommand{\indic}[1]{\1_{\{#1\}}}
\newcommand{\X}{\boldsymbol{X}}
\newcommand{\Y}{\boldsymbol{Y}}
\newcommand{\bl}{\boldsymbol{\lambda}}

\def\ignore#1{}

\def\Def{\ :=\ }

\def\bbP{\mathbb{P}}
\def\bbE{\mathbb{E}}

\def\bP{\mathbf{P}}
\def\bX{\X}

\def\bE{\mathbf{E}}

\def\parent#1{#1^{-1}}

\newcommand{\Gp}{\mc{G}'}
\newcommand{\nn}{\nonumber}
\newcommand{\vep}{\varepsilon}

\newcommand{\del}{\partial}

\newcommand{\1}{\mathbbm{1}}

\newcommand{\Xb}{{\X}^{\sss(\beta)}}
\newcommand{\xb}{{X}^{\sss(\beta)}}
\newcommand{\Xe}{{\X}^{\sss(\beta+\vep)}}
\newcommand{\xe}{{X}^{\sss(\beta+\vep)}}

\newcommand{\op}{\overline{\P}}
\newcommand{\zb}{{\boldsymbol Z}}

\title{
        On the speed of once-reinforced\\ biased random walk on trees
        }

\author{
Andrea Collevecchio\footnote{School of Mathematical Sciences, Monash  University, Melbourne.  E-mail {\tt
Andrea.Collevecchio@monash.edu}}, Mark Holmes \footnote{Department of Statistics, University of Auckland.  E-mail {\tt
m.holmes@auckland.ac.nz}}
\   and 
Daniel Kious\footnote{NYU-ECNU Institute of Mathematical Sciences at NYU Shanghai. E-mail {\tt
daniel.kious@nyu.edu}}
}

\begin{document}
\maketitle

 \begin{abstract}
We study the asymptotic behaviour of once-reinforced biased random walk (ORbRW) on Galton-Watson trees.  Here the underlying \Maaak{(unreinforced)} random walk has a bias towards or away from the root.  We prove that in the setting of multiplicative once-reinforcement the ORbRW can be recurrent even when the underlying  {biased random} walk is ballistic.  We also prove that, \Maaak{on Galton-Watson trees without leaves, the speed is positive in the transient regime}. {Finally, we prove that, on regular trees, the speed of the ORbRW is monotone decreasing in the reinforcement parameter when the underlying random walk has  high speed, and the reinforcement parameter is small.}
    \end{abstract}




\setcounter{footnote}{0}
\def\thefootnote{\arabic{footnote}}


\section{Introduction}
\label{sec-intro}
Reinforced random walks have been studied extensively since the introduction of the (linearly)-reinforced random walk of Coppersmith and Diaconis \cite{CD97}.  In this paper we study once-reinforced random walk, which was introduced by Davis \cite{Dav90} as a possible simpler model of reinforcement to understand.  While there have been recent major advances in the understanding of linearly reinforced walks on $\Z^d$ (see e.g.~\cite{ACK14,ST15,DST15,SZ15} and the references therein), rather less is known about once-reinforced walks on $\Z^d$.

In this paper we consider a class of random walks (on trees) characterised by two parameters $u_0,u_1\in (0,\infty)$.  The relative probability  of stepping to a given previously unvisited child compared to the probability of stepping to one's parent is $u_0/(1+u_0)$, while the relative probability of stepping to a given previously visited child as compared to the parent is $u_1/(1+u_1)$.   {Biased random walk corresponds to the case $u_0=u_1\ne 1$, while unbiased random walk corresponds to the case $u_0=u_1=1$.}

This class of walks includes once-reinforcement of biased random walks on Galton-Watson trees $\mc{G}$.  The underlying walk (when the reinforcement parameter is set to 0) can be described as follows:  when the walker is at $x\in \mc{G}$, each of the $\del(x)$ children of $x$ receive weight $\alpha$, while the parent of $x$ receives weight 1.  The walker chooses to step to a neighbouring vertex (i.e.~one of its children, or its parent if it has one) with probability proportional to the weights.  This walker has a drift (or bias) away from the origin on this step provided that $\alpha \del(x)>1$.  The once-reinforced biased random walk on $\mc{G}$ is a perturbation of this walk where vertices, or equivalently edges, that have been visited before have {higher weight}.  

{
Our main result, Theorem \ref{thm:speed}, gives a sharp criterion for the recurrence or transience (in fact, ballisticity) of these walks on Galton-Watson trees.
As a corollary we find examples of ballistic random walks that become recurrent after perturbation by a certain kind of once-reinforcement (see Corollary \ref{cor:Mult}).  

We also state a result of monotonicity for the speed of the once-reinforced random walk on regular trees when the bias is large and the reinforcement \Maaak{is }small enough, see Theorem \ref{thm:mono}.

The proof of Theorem \ref{thm:speed} uses two main ingredients which are, first, the adaptation of an idea of \cite{Co06}, and, second, the presence of a regular backbone in the Galton-Watson tree. The argument in \cite{Co06} relies on the embedding of a relevant Galton-Watson subtree. In order to properly use it here, we needed to introduce a family of coupled processes, called {\it extensions}, which in principle should allow one to use this strategy for a large class of processes.

The proof of Theorem \ref{thm:mono} is inspired by work of Ben Arous, Fribergh and Sidoravicius \cite{BAFS}, who proved a partial monotonicity result for biased random walks on Galton-Watson trees via a nice and natural coupling.  Here we needed to adapt this idea to find a coupling which works for certain {\em self-interacting} walks.  We have taken some care to extract the general features of the argument (which are relatively simple) and the details of the coupling for our particular setting (which are highly non-trivial).  This general method provides a coupling alternative to expansion techniques (e.g.~\cite{HH12}) when the self-interacting random walk has a large bias independent of the history.}

\subsection{The model}
\label{sec:themodel}
Consider a  rooted tree $\Gcal = (V, E)$ (where $V$ is the set of vertices and $E$ is the set of edges), augmented by adding  a parent $\parent{\r}$ to the root $\r$, the two being connected by an edge. If two vertices $\nu$  and $\mu$ are 
the endpoints of the same edge, they are said to be neighbours, and this property is 
denoted by $\nu \sim \mu$. 
The distance $\mathrm{d}(\nu,\mu)$ between any pair of  vertices $\nu, \mu$, not necessarily adjacent, 
is the number of edges in the unique self-avoiding path connecting $\nu$ to~$\mu$. We set $|\parent{\r}| = -1$. 
For any  other vertex $\nu$,  we let~$|\nu|$ 
be the distance from~$\nu$ to the root~$\r$. 
We use~$\parent{\nu}$ to denote the parent of~$\nu$, and $(\nu_i)_{i \in [\del(\nu)]}$, its children, where $\del(\nu)$ is the number of offspring of $\nu$ and  $[n]\Def \{1, 2, \ldots, n\}$.  Write $C_{\nu}=\{\nu_1,\dots, \nu_{\del(\nu)}\}$.  For $\mu\in V\setminus\{\r^{-1}\}$ we write $\nu < \mu$, 
if~$\nu$ is an ancestor of~$\mu$, i.e.~$\nu$ lies on the self-avoiding path connecting $\mu$ to $\r^{-1}$, and we write $\nu\le \mu$ if $\nu<\mu$ or $\nu=\mu$. If $\nu <\mu$ then $\mu$ is said to be a descendant of $\nu$.

{We are now about to define MAD walks, standing for {\it maximum acts differently}, in reference to their behavior on $\mathbb{Z}$}. Fix $\mc{G}=(V,E)$ and parameters $u_1,u_0>0$ and define \Maaak{the law $\bP^{\mc{G}}$ of } a MAD walk $\X = (X_n)_{n\in \Z_+}$, taking values on the vertices of $\Gcal$ (and taking nearest neighbour steps) as follows.  We set $X_0 = \r^{-1}$.  
Given $\mc{F}_n=\sigma(X_k:k\le n)$, we let $\mc{E}_\varnothing(n)=\{[X_{k-1},X_k]:k\le n\}\subset E$ denote the set of (undirected) edges crossed by time $n$.

For $\nu\ne \r^{-1}$ define
\begin{align}
W_n(\nu,\nu^{-1})&=1,
\end{align}
and for $i\in [\del(\nu)]$ define
\begin{align}
W_n(\nu,\nu_i)&=u_1\indic{[\nu,\nu_i]\in \mc{E}_\varnothing(n)}+u_0\indic{[\nu,\nu_i]\notin \mc{E}_\varnothing(n)}.
\end{align}


Note that when $X_n=\r^{-1}$, there is one child ($\del(\r^{-1})=1$) and no parent. So we set $\bP^{\mc{G}}(X_{n+1}=\r|\mc{F}_n)\indic{X_n=\r^{-1}}=\indic{X_n=\r^{-1}}$ (the walk always steps to $\r$ from $\r^{-1}$), a.s.   For $x\ne \r^{-1}$, the probability of stepping from $x$ to a given neighbour depends on both the number of previously visited children of $x$  and the total number of children $\del(x)$ of $x$. To be precise,
\begin{align} \label{transprob}
\bP^{\mc{G}}(X_{n+1}=\mu|\mc{F}_n)&:=\frac{W_{ n}(X_n,\mu)}{\sum_{\nu\sim X_n}W_{n}(X_n,\nu)}.
\end{align}

\blank{
Similarly, if $x$ has no child in $\Gcal$ ($\del(x)=0$) then one must always step to $x^{-1}$ from $x$, so $\bP^{\mc{G}}(X_{n+1}=X_n^{-1}|\mc{F}_n)\indic{\del(X_n)=0}=\indic{\del(X_n)=0}$.  Otherwise we define the transition probabilities of the process $\X$ 
via the relative probabilities
\begin{align}
&\frac{\bP^{\mc{G}}(X_{n+1}=X_{n,i}|\mc{F}_n)\indic{X_n \ne \r^{-1},\del(X_n)>0}}{\bP^{\mc{G}}(X_{n+1}=X_{n,i}|\mc{F}_n)+\P(X_{n+1}=X_{n,i}|\mc{F}_n)}\\
&\qquad=\big(b\indic{X_{n,i}\notin \mc{E}_\varnothing(n)}+a\indic{X_{n,i}\in \mc{E}_\varnothing(n)} \big)\indic{X_n \ne \r^{-1},\del(X_n)>0}.
\label{relative}
\end{align}
From this it is easy to see that all previously unvisited children are equally likely to be stepped to, and all previously visited children are equally likely to be stepped to (but that these two probabilities agree only if $a=b$).  
\[\bP^{\mc{G}}(X_{n+1}=X_n^{-1}|\mc{F}_n)=\left(K_\varnothing(n)\frac{a}{1-a}+(\del(X_n)-K_\varnothing(n))\frac{b}{1-b}+1\right)^{-1}.\]
The above is equivalent to a 2-parameter family of models where (directed) edges are weighted and walkers step with probabilities proportional to the weights of the edges.  Let us refer to the edge to one's parent as the backward edge and edges to children as forward edges.  Then we can set the weight of the backward edge as 1, the weight of a previously used forward edge as $a'$ , and the weight of a previously unused forward edge as $b'$.  Then $a=a'/(1+a')$ and $b=b'/(1+b')$, or in other words, $a'=a/(1-a)$ and $b'=b/(1-b)$.
}

In the following three examples we define walks whose step probabilities are proportional to weights on edges incident to the current location.  
\begin{EXA}
\label{exa:unreinforced}
Let $\alpha>0$ and suppose that forward edges each have weight $u_1=u_0=\alpha$.  This is a biased (when $\alpha \ne 1$) random walk on the tree.
\end{EXA}

\begin{EXA}[Multiplicative once-reinforcement]
\label{exa:mult}
Let $\alpha,\beta>0$, and suppose that the backward edge has weight $(1+\beta)$ and forward edges that have not been previously visited have weight $\alpha$, while forward edges that have previously been visited have weight $\alpha(1+\beta)$.  After rescaling all weights so that the backward edge has weight 1, this is equivalent to $u_1=\alpha$ and $u_0=\alpha/(1+\beta)$.
\end{EXA}

\begin{EXA}[Additive once-reinforcement]
\label{exa:add}
Let $\alpha, \beta>0$, and suppose that the backward edge has weight $(1+\beta)$ and forward edges that have not been previously visited have weight  $\alpha$, while forward edges that have previously been visited have weight  $(\alpha+\beta)$.  After rescaling we get $u_1=(\alpha+\beta)/(1+\beta)$ and $u_0=\alpha/(1+\beta)$.
\end{EXA}

Multiplicative and additive once-reinforcement coincide only when there is no intrinsic bias (i.e.~when $\alpha=1$).  Both are in some sense natural notions of once-reinforcement when the underlying walk has a drift.  Multiplicative reinforcement has the property that if all neighbours of $\mu\ne \r^{-1}$ have been visited by the walk, then the step probability from $\mu$ is that of an unreinforced random walk (i.e.~the reinforcement vanishes at such locations).  Additive reinforcement is the situation that gives the natural criterion for transience versus recurrence on regular trees (based on the sign of $\alpha d-1$, where $d$ is the number of children of each vertex).


Thus far we have described the dynamics of the random walk on a fixed tree $\mc{G}$.  We can also consider the behaviour of the model averaged over a family of {\it random} trees.  Let $P$ denote the law of a Galton-Watson tree $\mc{G}$, and let $d\in (0,\infty)$ denote the mean number of offspring of each individual.  We use the notation $\P()$ to denote the averaged law of a MAD walk on a random tree $\mc{G}$ with law $P$, defined by 
\begin{align}
\label{annealed}
\P(\mc{G}\in A, \X\in B):=\int_A \bP^{\mc{G}}(\X\in B) dP.
\end{align}


Our first main result is the following theorem which reveals a non-trivial (in the case of multiplicative reinforcement) phase transition for recurrence/transience.  Sidoravicius has conjectured that there is a phase transition for transience for once-reinforced simple symmetric random walk on $\Z^d$, $d\ge 3$.  In \cite{KS16}, a phase transition is shown for the (non-biased) once-reinforced random walk on $\Z^d$-like trees, which are inhomogeneous trees with polynomial growth. 
It should be noted that, in both \cite{KS16} and the conjecture of Sidoravicius, although the underlying walk is transient, it always has zero-speed. Here, we provide examples where the underlying walk has positive speed, but its once-reinforced counterpart is recurrent.

In order to state the result, let $G_\infty$ denote the event that $\mc{G}$ survives forever.  Clearly on $(G_\infty)^c$ the walker is recurrent under the so-called {\it quenched} measure $\bP^{\mc{G}}$.  Therefore our main results are stated for the supercritical setting $d>1$, whence $\bbP(G_\infty)>0$.  For positivity of the speed we assume that every individual has at least one child, i.e.
\begin{align}
P(\del(\r)\ge 1)=1.\label{no_leaves}
\end{align}
This is equivalent to saying that we restrict ourselves to the subclass of (rooted) trees with no leaves (apart from $\r^{-1}$).  \\

Here, we say that $\X$ is transient if it comes back to its starting point finitely often almost surely, and we say that $\X$ is recurrent if it comes back to its starting point infinitely often almost surely. Note that  (see Section \ref{sec:green}) these notions of recurrence and transience are unambiguous and match any of the classical definitions.  In particular, a walk is either transient or recurrent.  


\begin{THM}
\label{thm:speed}
Fix $u_1,u_0>0$,   and $d>1$.  {Let $\mc{G}$ be a Galton-Watson tree with mean offspring $d$ and let $G_\infty$ be the event that it survives forever.} Then on the event $G_\infty$, the following occur $\P$-almost surely:
\begin{itemize}
\item[$(1)$] If $du_0>1-u_1+u_0$ then $\X$ is transient;
\item[$(2)$] If $du_0<1-u_1+u_0$ then $\X$ is recurrent.
\end{itemize}
Moreover if \eqref{no_leaves} holds and $u_1\ge u_0$ then the speed is (deterministic and) positive in $(1)$.
\end{THM}

{
\begin{REM}
In Theorem \ref{thm:speed}, we only consider the case $u_1\ge u_0$ for the positivity of the speed. This corresponds to the case where the interaction is attractive, and this fact is used, e.g., in Lemma \ref{loct}. We believe that similar results can be obtained if $u_1<u_0$ but the technique needs to be adapted.
\end{REM}
}

\begin{REM}
\Maaak{We believe that on Galton-Watson trees, when $du_0=1-u_1+u_0$ the MAD walk is recurrent. This should not be true in general (i.e.~if we do not assume that the tree is a Galton-Watson tree).} 
\end{REM}
\begin{REM}
\Maaak{The criterion in Theorem \ref{thm:speed} can be written as $(d-1)u_0+u_1\lessgtr 1$.
Thus, one can view the criterion as saying that the walker is recurrent (resp.~transient) if it has a bias towards (resp.~away from) the root {\em when at a site where exactly one child has been visited.}}
\end{REM}

\begin{REM}
\Maaak{
One may ask if the walk is {\em positive recurrent} when $du_0<1-u_1+u_0$.  We believe that in this setting the expected return time for the $k$-th return to the root is finite for each $k$, but that if $1<du_1$ then these expected times diverge  with $k$.}
\end{REM}

Putting in the values of $u_1$ and $u_0$ from Examples \ref{exa:mult} and \ref{exa:add} we obtain the following:

\begin{COR}[Multiplicative once-reinforcement]
\label{cor:Mult}
Fix $\alpha>0$ and $\beta>-1$ and $d>1$ in Example \ref{exa:mult}.  Then on the event $G_\infty$, the following occur $\P$-almost surely:
\begin{enumerate}
\item [$(1)$]
If $d\alpha >1+\beta - \alpha\beta$ then $\X$ is transient;
\item [$(2)$]
If $d\alpha < 1+\beta - \alpha\beta$ then $\X$ is recurrent.
\end{enumerate}
Moreover if \eqref{no_leaves} holds then the speed is (deterministic and) positive in $(1)$.
\end{COR}

\begin{COR}[Additive once-reinforcement]
\label{cor:add}
Fix $\alpha>0$ and $\beta>-\min(\alpha,1)$ and $d>1$ in Example \ref{exa:add}.  Then on the event $G_\infty$, the following occur $\P$-almost surely:
\begin{enumerate}
\item[$(1)$]
If $d\alpha>1$ then $\X$ is transient;
\item [$(2)$]
If $d\alpha<1$ then $\X$ is recurrent.
\end{enumerate}
Moreover if \eqref{no_leaves} holds then the speed is (deterministic and) positive in $(1)$.
\end{COR}

One might be surprised to see the phase transition at \Maaak{$d\alpha >1+\beta - \alpha\beta$} rather than at \Maaak{$d\alpha=1$} in the multiplicative reinforcement regime, but that surprise dissipates somewhat after considering the following heuristic, in which we assume that $\beta>0$ and $\mc{G}$ is a regular tree:

\Maaak{If $d\alpha<1$ then $\X$ will clearly be recurrent as it will always have a drift to the root regardless of the local environment (even if all $d$ children of the current location are reinforced, their  total weight $d\alpha(1+\beta)$ is still less than the weight of the parent $(1+\beta)$).  However it is still plausible that $\X$ is recurrent if $d\alpha$ is slightly larger than 1, since the environment seen by the walker won't always have all children reinforced.  


Similarly, if $1+\beta <d\alpha$ then $\X$ will clearly be transient as we have a drift away from the root regardless of the local environment (even when none of the children of the current location are reinforced).  It is still plausible that $\X$ is transient if $d\alpha$ is slightly smaller than $1+\beta$, since the environment seen by the walker won't always have no children reinforced.  }

Note that when $\alpha<1$,
\[1<1+\beta -\alpha\beta <1+\beta,\]
so Corollary \ref{cor:Mult} affirms that \Maaak{when $d\alpha$ is a bit larger than 1, there is enough push towards the root for the walker when at frontier sites (sites where at least one child has not been visited before) to make it recurrent.  It also affirms that when $d\alpha$ is a bit smaller than $1+\beta$ there is enough push away from the root for the walker when at sites when some child has been visited to make it transient.} 

{For additive reinforcement, perhaps surprisingly, the reinforcement parameter $\beta$ plays no role in the recurrence or transience of the process.} Nevertheless, a similar argument applies, with the relevant pair of bounds being $d(\alpha+\beta)<1+\beta$ implies drift toward the root and $d\alpha>1+\beta$ drift away.  Now when $d>1$ and $\beta>0$,
\[1-(d-1)\beta  <1<1+\beta.\]
Thus, Corollary \ref{cor:add} of the theorem affirms that visits to \Maaak{locations where at least one but not all children of the current location have been reinforced generate enough push towards the root to retain recurrence when $d\alpha$ is a bit bigger than $1-(d-1)\beta$, and enough push away from the root to retain transience when $d\alpha$ is a bit smaller than $1+\beta$.}

Note that our results state that the speed is positive in the transient regime on trees without leaves.  In general this is not expected to hold for trees with leaves. Indeed, in \cite{LPP}, it is proved that biased random walks on Galton-Watson trees with leaves are transient but with a vanishing speed as soon as the bias is large enough (and in particular the speed is not monotone). For a study of the speed of random walks defined on Galton-Watson trees, see e.g. \cite{Aid1}.

For fixed $d$, as long as the tree has no leaves (e.g.~a $d$-regular tree) then it is natural to expect that the speed of the walker is increasing in $u_0$ and $u_1$.
\begin{CON}
\label{con:mono}
On $d$-regular trees the speed given in Theorem \ref{thm:speed} is monotone increasing in both $u_1$ and $u_0$.  It is strictly monotone in the transient regime.
\end{CON}
This is simple to prove when $d=1$ (i.e.~on $\Z$).  We do not know how to prove this result when $d>1$.

Note that in Example \ref{exa:mult}, $u_1$ does not depend on $\beta$ and $u_0$ is decreasing in $\beta$.  In Example \ref{exa:add}, $u_1$ is {\it increasing} in $\beta$ when $\alpha<1$.

By adapting a coupling argument of \cite{BAFS}, we can prove that the speed is monotone decreasing in the reinforcement for high biases and small reinforcements.  This is the content of our second main result.

\begin{THM}
\label{thm:mono}
For every $\alpha,d$  such that $\alpha d \ge 150$ there exists $\beta_0(\alpha,d)>0$ such that the speed $v(\beta)$ of  multiplicative once-reinforced random walk on $d$-regular trees  is {(strictly)} monotone decreasing in $\beta$ for all $\beta\in [0,\beta_0]$.  
\end{THM}

\begin{REM}
Here, we do not work on improving the lower bound $\alpha d\ge150$, but, in Section \ref{sectimprove}, we explain how this bound can be turned into $\alpha d\ge22$ (and probably less).  This improvement arises from looking at the coupling of Section \ref{sec:coupling} in finer detail, requiring more involved and precise computations.
\end{REM}
\begin{REM}
Note that  Theorem \ref{thm:mono} concerns only multiplicative reinforcement. One could obtain a similar result for additive reinforcement by slightly adapting the proof we give here: in Section \ref{proofadd}, we provide instructions for how to do this.
\end{REM}

The remainder of the paper is organised as follows.  In Section \ref{sec:speed} we prove Theorem \ref{thm:speed}.  In Section \ref{sec:mono} we prove Theorem \ref{thm:mono}.


\section{Proof of Theorem \ref{thm:speed}}
\label{sec:speed}
In this section we prove Theorem \ref{thm:speed}.  Our immediate objective will be defining a probability space on which we have an appropriate walk $\X=\X(\mc{G})$ defined on each tree $\mc{G}$ as well as a family of walks $\X^{\sss(\Gp)}=\X^{\sss(\Gp)}(\mc{G})$ on certain subtrees $\Gp$ of $\mc{G}$.  Moreover, our probability space will be chosen to allow different starting ``environments/configurations'', $\omega$, corresponding to the possible configurations of visited edges that a walker could see in finite time.  The law of a MAD walk $\X$ on a fixed tree $\mc{G}$ with configuration $\omega$ started at $\r^{-1}$ will be denoted by $\bP^{\mc{G}}_\omega$.

Let us define an infinite-branching tree (or multi-index set) $\mc{M}$, that will contain every tree $\mc{G}$ with vertices of finite degree.  

We write $\alpha_{\{-1\}}=\r^{-1}$ for the unique element of $\mc{M}$ of generation $-1$ and $\alpha_{\{0\}}=\r$ denotes the unique child of $\alpha_{\{-1\}}$ (i.e.~the unique element of $\mc{M}$ of generation 0).  For $n\ge 1$, and $\alpha_{\{1\}},\dots, \alpha_{\{n\}} \in \N$, each multi-index ${\alpha}=\alpha_{\{-1\}}\alpha_{\{0\}}\alpha_{\{1\}}\dots \alpha_{\{n-1\}}\alpha_{\{n\}}$ of generation $n$ is the $\alpha_{\{n\}}$-th child of its parent ${\alpha^{-1}}=\alpha_{\{-1\}}\alpha_{\{0\}}\alpha_{\{1\}}\dots \alpha_{\{ n-1\}}$.  Let $\mc{M}=\{\alpha_{\{-1\}}\alpha_{\{0\}}\alpha_{\{1\}}\dots \alpha_{\{n-1\}}\alpha_{\{n\}}:n \in \Z_+, \alpha_1,\dots, \alpha_n \in \N\}$ denote the set of such multi-indices.  For $\nu, \mu \in \mc{M}$ we write $\nu\sim \mu$ if $\nu=\mu^{-1}$ or $\mu=\nu^{-1}$.  Write $|\nu|$ for the generation of $\nu$, and write $\nu\le \mu$ if $|\nu|\le |\mu|$ and $\mu_{\{i\}}=\nu_{\{i\}}$ for each $i\le |\nu|$.  Let $E_\mc{M}=\{[\nu, \mu]  \in \mc{M}:\nu\sim \mu\}$ be the set of (unordered) edges of $\mc{M}$.  For simplicity of notation, we also write $(\nu_i)_{i \in \N}$ for the children of $\nu \in \mc{M}\setminus \{\r^{-1}\}$.

Define $\Ccal$ to be the collection of maps $\omega \colon E_{\mc{M}} \rightarrow \{0, 1\}$ that satisfy the following two properties: 
\begin{itemize}
\item[(i)]  For any $\nu \in \mc{M}\setminus\{\r^{-1}\}$,  we have  
 $$  \exists i \in \N \mbox{ such that } \omega(\nu, \nu_i) = 1 \qquad \implies  \qquad \omega(\parent{\nu}, \nu) = 1;$$
\item[(ii)] $\Gamma_{\omega}=\omega^{-1}(1)\equiv \{e\in E_{\mc{M}}:\omega(e)=1\}$ is a finite set.  
\end{itemize}
We call $\Ccal$ the space of coherent configurations.  Define the configuration $\omega^* \in \Ccal $ by
\begin{equation}
\omega^*(e)=\begin{cases}
1, & \text{ if } e=(\r^{-1}, \r),\\
0, & \text{ otherwise}.
\end{cases}
\end{equation}


Fix a subtree $\mc{G}\subset \mc{M}$ with finite degrees.  We define the law $\bP_\omega^{\mc{G}}$ of a walk $\X$ on $\mc{G}$ as follows.  Set $X_0 = \r^{-1}$ almost surely.  Given $\mc{F}_n=\sigma(X_k:k\le n)$ and $\omega \in \mc{C}$ we 
define  $\mc{E}_{\omega}(n)=\Gamma_{\omega}\cup \mc{E}_\varnothing(n)$, where we recall that $ \mc{E}_\varnothing(n)$ is the set of edges crossed at time $n$.
If $\nu\ne \r^{-1}$ define
\begin{align}
W_n(\nu,\nu^{-1})&=1,
\end{align}
and for $i\in [\del(\nu)]$ define
\begin{align}
W_n(\nu,\nu_i)&=u_1\indic{[\nu,\nu_i]\in \mc{E}_\omega(n)}+u_0\indic{[\nu,\nu_i]\notin \mc{E}_\omega(n)}.
\end{align}
For $\mu\sim X_n$, define
\begin{align}
\bP^{\mc{G}}_\omega(X_{n+1}=\mu|\mc{F}_n)&:=\frac{W_n(X_n,\mu)}{\sum_{\nu\sim X_n} W_n(X_n,\nu)}.
\end{align}
In particular, when $X_n=\r^{-1}$ there is one child ($\del(\r^{-1})=1$) and no parent so $\bP^{\mc{G}}_\omega(X_{n+1}=\r|\mc{F}_n)\indic{X_n=\r^{-1}}=\indic{X_n=\r^{-1}}$ (the walk always steps to $\r$ from $\r^{-1}$).  
In the above we have set $X_0=\r^{-1}$.  If we instead set $X_0=\nu$, for some $\nu\in V$, then we write the law of the process 
above as $\bP^{\Gcal}_{\omega,\nu}$ and expectation with respect to $\bP^{\Gcal}_{\omega, \nu}$ will be denoted by 
$\bE^{\Gcal}_{\omega,\nu}$.  In particular, $\bP^{\Gcal}_{\omega, \r^{-1}}=\bP^{\Gcal}_{\omega}$.  The corresponding averaged measures are denoted by $\P_{\omega, \nu}$ and $\P_{\omega}$, with the former being a conditional measure (conditional on $\nu\in \mc{G}$).

When we deal with the canonical configuration $\omega^*$, we may drop it from the notation.  This is consistent with the notation defined in Section \ref{sec:themodel}.

\subsection{A coupling of walks and environments}
\label{sec:coupling1}


Let $(\Omega, \mc{F},\P)$ denote a probability space on which ${\bf Y}=(Y(\nu,\mu,k): (\nu,\mu) \textrm{ in } \mc{M}^2, \mbox{with $\nu \sim \mu$}, \textrm{ and }k \in \N)$ is a family of independent mean 1 exponential random variables (here $(\nu,\mu)$ is considered an {\it ordered} pair); and $\mc{G}\subset \mc{M}$ is a (supercritical) Galton-Watson tree satisfying \eqref{no_leaves}, that is independent of ${\bf Y}$.

For any vertex $\nu\ne \r^{-1}$, write $\nu_0=\nu^{-1}$ {and recall that $\nu_1,\dots,\nu_{\partial(\nu)}$ denote the children of $\nu$}.  
For $\bar{k}=(k_0,\dots,k_{\del(\nu)}) \in \N^{\del(\nu)+1}$, let $A_{\bar{k},n,\nu}=\{X_n = \nu\}\cap\bigcap_{0\le s\le\del(\nu)} \{\#(1\le j \le n \colon [X_{j-1},X_j] = [\nu,\nu_s]) = k_s\}$. Define the quantities
\begin{align} \label{wj1}
w_j(\nu,\nu^{-1})&=1,\, \forall j\ge0,\\ \label{wj2}
w_0(\nu,\nu_i)&=u_0\indic{[\nu,\nu_i]\notin \Gamma_\omega}+u_1\indic{[\nu,\nu_i]\in \Gamma_\omega},\ \ \forall i\in [\del(\nu)]\\ \label{wj3}
w_j(\nu,\nu_i)&=u_1,\, \forall j\ge1, \ \forall i\in [\del(\nu)].
\end{align}
Then on the event 
\begin{align}\label{ursula0}
 A_{\bar{k},n,\nu}\cap \left\{\sum_{i=0}^{k_j}\frac{  Y(\nu, \nu_j, i)}{w_i(\nu, \nu_j) } = \min_{s \in \{0, 1, \ldots, \del(\nu)\}}\Big\{\sum_{i=0}^{k_s}\frac{Y(\nu, \nu_s, i)}{w_i(\nu, \nu_s)} \Big\}\right\}, 
 \end{align}
 we set $X_{n+1} = \nu_j$. It is easy to check, from properties of independent exponential random variables and the memoryless property, that this provides a construction of a MAD walk on $\mc{G}$ (with law $\bP^{\mc{G}}_\omega$). This continuous-time embedding is classical: it is called {\it Rubin's construction} and  can be found in \cite{Dav90}, for instance.

\begin{DEF}\label{spect} A subtree $\Gp = (V', E')$ of $\mc{G}$ is said to be {\it special} if the vertex $\r'$ in $V'$ with minimal distance from $\r^{-1}$ has degree one in $\Gp$.  Let $\mc{S}(\mc{G})$ denote the set of special subtrees of $\mc{G}$.
\end{DEF}


Fix $\mc{G}$, and define the MAD walk $\X$ on $\mc{G}$ as above.  For $\Gp = (V', E')\in \mc{S}(\mc{G})$, let $\tau_1=\tau_1(\Gp)=\inf\{m\in \N:[X_{m-1},X_m]\in E'\}$.  For $n\ge 2$ recursively define $\tau_n=\tau_n(\Gp)=\inf\{m>\tau_{n-1}(\Gp):[X_{ m-1},X_{m}]\in E'\}$ (this is the collection of times that $\X$ walks on $E'$).  Then we define a walk $\X^{\sss (\Gp)}=(X^{\sss (\Gp)}_n)_{n\ge 0}$ on $\Gp$ as follows:

Set $\X^{\sss (\Gp)}_0=\r'$, and 
\begin{align}
\X^{\sss (\Gp)}_n=\X_{\tau_n}
\end{align} 
for all strictly positive $n\le n_\infty:=\inf\{n: \tau_n=\infty\}-1$.  If $n_\infty$ is infinite then this defines the entire walk $\X^{\sss (\Gp)}$.  In particular it is easy to see that $\X^{\sss (\mc{G})}=\X$.  Otherwise $n_\infty$ is finite and we continue to generate the additional steps of $\X^{\sss (\Gp)}$ according to the unused $Y$ variables as follows:

For $n\ge n_\infty$ and $\bar{k}=(k_\mu)_{\mu: [\nu,\mu]\in E'} \in \N^{|\mc{G}'\cap C_{\nu}|+1}$, let $A^{\sss (\Gp)}_{\bar{k},n,\nu}=\{X^{\sss (\Gp)}_n = \nu\}\cap\bigcap_{\mu: [\nu,\mu]\in E'} \{\#(1\le j \le n \colon [X^{\sss (\Gp)}_{j-1},X^{\sss (\Gp)}_j] = [\nu,\mu]) = k_\mu\}$.  For $\nu'$ such that $[\nu, \nu']\in E'$, on the event 
\begin{align}\label{ursula}
A^{\sss (\Gp)}_{\bar{k},n,\nu}\cap \left\{\sum_{i=0}^{k_{\nu'}}\frac{Y(\nu, \nu', i)}{w_i(\nu, \nu')} = \min_{\mu: [\nu,\mu]\in E'}\Big\{\sum_{i=0}^{k_{\mu}}\frac{Y(\nu, \mu, i)}{w_i(\nu, \mu)} \Big\}\right\}, 
\end{align}
 we set $X^{\sss (\Gp)}_{n+1} = \nu'$, where the $w_i$'s are defined in \eqref{wj1}, \eqref{wj2} and \eqref{wj3}.   

It is immediate from this definition that the steps taken by $\X$ on the edges $E'$ of $\Gp$ up to time $\tau_{n_{\infty}}$ are exactly those taken by $\X^{\sss (\Gp)}$ up to time $n_{\infty}$, that the constructions are consistent (i.e.~\eqref{ursula} is a continuation of \eqref{ursula0} for $n\ge n_\infty$) and that after time $\tau_{n_\infty}$ the walk $\X$ never walks on $E'$.  

Note that the steps of $\X^{\sss (\Gp)}$ from leaves of $E'$ are deterministic and the steps along the edges of $E'$ are determined {\it only} by the clocks $Y(\cdot,\cdot,\cdot)$ attached the edges of $E'$.  From this it is easy to see that for $\Gp,\mc{G}''\in \mc{S}$, the walks $\X^{\sss (\Gp)}$ and $\X^{\sss (\mc{G}'')}$ are independent if:  {\it whenever $[e_1,e_2]\in E'\cap E''$, we have that each of $e_1$ and $e_2$ is a leaf of either $\Gp$ or $\mc{G}''$}.   
Moreover this pairwise independence extends to independence of any countable collection of such walks with such overlaps.  

We call $\X^{\sss (\Gp)}$ the {\it extension} of $\X$ on $\Gp$.  Of particular interest will be the case where $\Gp=[ \nu,\mu]$ is the unique self-avoiding path connecting $\nu$ to $\mu$, for some $\nu,\mu\in\mc{G}$ such that $\nu<\mu$.  In this case we write $\X[ \nu,\mu]$ for $\X^{\sss([ \nu,\mu])}$ and we denote $\bP_\omega^{[\nu,\mu]}$ its associated measure.


\subsection{Walks on paths}
In this section we describe various walks on paths that are important for understanding walks $\X[\nu, \mu]$ (and hence also $\X$).

\begin{DEF}
Given parameters $a,b\in (0,1)$ and $\ell\in\Z_+$, we call a nearest neighbour random walk $\boldsymbol{Z}$ on $\Z_+$ with natural filtration $(\mc{F}_n)_{n \in\Z_+}$ an $(a,b,\ell)$-MAD (maximum acts differently) walk if $Z_0=\ell$ almost surely and 
\[\P(Z_{n+1}=Z_n+1|\mc{F}_n)=\begin{cases}
b, & \textrm{ if } Z_n=\max_{k\le n}Z_k\\
a, & \textrm{ otherwise}.
\end{cases}\]
\end{DEF}
\begin{LEM}
\label{lem:pam}
Fix $a,b\in(0,1)$, $\ell\ge0$ and let $\boldsymbol{Z}$ be an $(a,b,\ell)$-MAD walk.  Then 
\begin{equation}
\phi_{\ell,n}:=\P(\boldsymbol{Z} \textrm{ hits }n \textrm{ before } -1)=
\begin{cases}
\prod_{j=\ell}^{n-1}\frac{b(j+1)}{bj+1}, & \text{ if }a=\frac{1}{2}\\
\prod_{j=\ell}^{n-1}\frac{b-b\zeta^{j+1}}{b-\zeta^{j+1}(1-\frac{1-b}{\zeta})}, & \text{ otherwise,}
\end{cases}
\end{equation}
where $\zeta=(1-a)/a$.
\end{LEM}
\proof 
Suppose that the walk $\boldsymbol Z$ has reached level $j$ for the first time (before  reaching -1).  
Then with probability $b$ it reaches $j+1$ on the next step.    With probability $1-b$  it steps to $j-1$.  If this happens then via the classical gambler's ruin problem we see that the probability $a_j$ of hitting $j$ again before -1 is given by 
\[a_j=\begin{cases}
\frac{j}{j+1}, & \text{ if }a=\frac{1}{2}\\
\frac{\sum_{r=0}^{j-1}\zeta^{r}}{\sum_{m=0}^j \zeta^m}=\frac{1-\zeta^j}{1-\zeta^{j+1}}, & \text{ otherwise}.
\end{cases}\]
It follows that the probability that the walk $\boldsymbol Z$, having reached $j$ for the first time,  reaches $j+1$ before $-1$ is the solution $\eta_j$ to 
\begin{align}
\eta_j=b+(1-b)a_j\eta_j.
\end{align}
Solving gives
\begin{align}
\eta_j=\frac{b}{1-(1-b)a_j}=\begin{cases}
\frac{b(j+1)}{bj+1}, & \text{ if }a=\frac{1}{2}\\
\frac{b-b\zeta^{j+1}}{b-\zeta^{j+1}(1-\frac{1-b}{\zeta})}, & \text{ otherwise. }
\end{cases}
\end{align}
The result follows from the fact that in order to reach $n$ before hitting -1 we must reach $j+1$ from $j$ (before hitting -1) for each $j=\ell,\dots,n-1$.
\qed

\medskip

Next, we consider the original MAD $\X$ on $\Gcal$. Recall that $\X[\nu, \mu]$ is a walk on the interval $[\nu,\mu]\subset \mc{G}$, starting from $\nu$.  Let $S_\mu=\inf\{n\ge 1:X_n[\nu, \mu]=\mu\}$ be the first hitting time of $\mu$ and $S_\nu=\inf\{n\ge 1:X_n[\nu, \mu]=\nu\}$ be the first return time to $\nu$ for this walk.  The following Corollary follows immediately from Lemma \ref{lem:pam} with $a=u_1/(1+u_1)$ and $b=u_0/(1+u_0)$ (and therefore $\zeta=u_1^{-1}$).
\begin{COR}
\label{cor:path}
 Fix a vertex $\mu\in \mc{G}$ at level $n$. Then
$$\psi_n:=\bP^{[\r^{-1},\mu]}_{\omega^*}\left(S_{\mu}<S_{\r^{-1}}\right)=\begin{cases}
 \prod_{j=0}^{n-1} \frac{u_0(j+1)}{u_0(j+1)+1}, & \text{ if }u_1=1,\\
 \prod_{j=0}^{n-1} \frac{u_0(u_1^{j+1}-1)}{u_0u_1^{j+1}+u_1-u_0-1}, & \text{ otherwise}.
\end{cases}
$$
{Moreover,  whenever $\mu$ is a descendent of $\nu$ with tree distance $n$ between them, we have that $\bP^{[\nu,\mu]}_{\omega^*,\Maaak{\nu}}\left(S_{\mu}<S_{\nu}\right)=\bP^{[\nu,\mu]}_{\omega^*,{\nu'}}\left(S_{\mu}<S_{\nu}\right)=\psi_n$, where $\nu'$ is the neighbour of $\nu$ in $[\nu,\mu]$.}
\end{COR}

\proof The probability in question is nothing more than the probability that an $(a,b,0)$-MAD walk $\boldsymbol{Z}$ hits $n$ before -1, with $a=u_1/(1+u_1)$ and $b=u_0/(1+u_0)$ (and therefore $\zeta=u_1^{-1}$).  Hence the result is immediate from Lemma \ref{lem:pam}.
\qed


For fixed $u_1,u_0>0$ we introduce a partial ordering on $\Ccal$ as follows. 
\begin{DEF}\label{DEF:2.5}
Fix $u_0,u_1>0$.  For $\omega, \omega' \in \Ccal$, we write $\omega \ge  \omega'$  if either:
\begin{itemize}
\item[(i)] $u_1\ge u_0$, and $\omega\ge \omega'$ pointwise (i.e.~$\omega^{-1}(1)\supset {\omega'}^{-1}(1)$), or;
\item[(ii)] $u_1\le u_0$, and $\omega\le \omega'$ pointwise (i.e.~$\omega^{-1}(1)\subset {\omega'}^{-1}(1)$).
\end{itemize}
We say that $\omega > \omega'$ if  $\omega \neq \omega'$ and $\omega \ge \omega'$. 
\end{DEF}

Then we have the following natural result (that is proved e.g.~in \cite{HS12}).
\begin{LEM}
\label{lem:pathmono}
Let $\mu$ be a descendent of $\nu$, and $\omega\ge \omega'$ (both in $\mc{C}$).  Then on the probability space of Section \ref{sec:coupling1}: if $S_{\mu}<S_{\nu}$ in environment $\omega'$ then also $S_{\mu}<S_{\nu}$ in $\omega$.
\end{LEM}
Above we have used the notation $S_\nu$ to denote the first non-zero hitting time of $\nu$ in the context of walks on paths.  In preparation for proving transience and recurrence properties of $\X$ we introduce the following hitting time notation.  
\begin{DEF}
For each vertex \Maaak{$\nu\in \mc{G}$}, define $T(\nu)=\inf\{n\ge 1:X_n=\nu\}$ to be the first (strictly positive) hitting time  of $\nu$ by the process $\X$. If $\nu$ is never hit after time $0$ then we set $T(\nu) =\infty$.  
\end{DEF}
\subsection{Proof of transience and recurrence}
\label{sec:green}
{Let us first clarify that the notions of recurrence and transience are well-defined.} 
\Maaak{
\begin{LEM}
Fix $\omega=\omega^*$ and $\mathcal{G}$.
\begin{itemize}
\item[(i)] If  $\bP^{\mc{G}}(T(\parent{\r}) <\infty)=1$ then $\bP^{\mc{G}}(\cap_{\nu \in \mc{G}}\{X_n =\nu \text{ infinitely often}\})=1$.
\item[(ii)] If $\bP^{\mc{G}}(T(\parent{\r}) <\infty)<1$ then $\bP^{\mc{G}}(\cup_{\nu \in \mc{G}}\{X_n =\nu \text{ infinitely often}\})=0$.
\end{itemize}
\end{LEM}
\proof First, note that for every $\nu,\mu\in \mc{G}$, there exists an $\eta>0$ such that on every visit to $\nu$ there is probability at least $\eta$ of hitting $\mu$ before returning to $\nu$.  Thus, the events  $E_{\infty}:=\{\textrm{every vertex in $\mc{G}$ is visited infinitely often}\}$ and $S_{\infty}:=\{\textrm{some vertex in $\mc{G}$ is visited infinitely often}\}$ satisfy 
\begin{align}
\label{mm1}E_\infty\subset S_\infty, \textrm{ and } \bP^{\mc{G}}(S_\infty\setminus E_{\infty})=0.
\end{align}
Next, if the walk $\X$ comes back to $\r^{-1}$ almost surely then, since any finite sequence of nearest neighbour steps is realised with positive probability by the walk, we have that the  probability of returning infinitely often to $\r^{-1}$  is 1, and therefore the probability of visiting every vertex infinitely often is 1.   This proves (i).   

Now suppose that with probability $\vep>0$ the walk $\X$ does not return to $\r^{-1}$,   then at every level $n\ge1$ of $\mathcal{G}$, there exists a vertex $\nu_n$ from which $\X$ escapes to infinity (without ever visiting $\r^{-1}$ again) after the first visit to $\nu_n$ with probability at least $\vep$. Now, if $\X$ comes back to $\r^{-1}$ infinitely often, it reaches infinitely many of the $\nu_n$ and thus eventually escapes to infinity which contradicts the fact that $\X$ comes back to $\r^{-1}$ infinitely often. Therefore, it will almost surely return to $\r^{-1}$ only a finite number of times.  This proves (ii).
\hfill
\qed
}


\medskip

We continue to fix $\omega=\omega^*$, but now $\mc{G}$ is a Galton-Watson tree under the annealed measure $\P$, with (annealed) mean offspring $d$. \\
In order to prove Theorem \ref{thm:speed}(1) we need the following result.
\begin{PRP}\label{propsansnom}
If $\gamma:=\P(T(\parent{\r}) = \infty) >0$, we have that
\[
\P\left(\lim_{n\to\infty}|X_n|=\infty\;
\Big|\; G_\infty\right)=1,
\]
where $\P$ is the annealed law of the MAD walk on GW trees. 
\end{PRP}
\begin{proof}
Let $H=\{\r^{-1} \textrm{ is visited infinitely often}\}$.  Due to \eqref{mm1}, it is sufficient to show that if $\gamma>0$ then $\P(H,G_\infty)=0$.  For this argument it is convenient to realise the process $\X$ under $\P$ as follows:  Whenever we reach a new (i.e.~previously unvisited) vertex $x\in \mc{G}$, we reveal the children $x_1,\dots, x_{\del(x)}$ of $x$.  Thus when we reach a site for the first time, its progeny has not yet been seen (and has the same Galton-Watson law as the progeny of $\r$).

Set $T^*_1 = 1$. 
Denote by $T^*_n$ the times when $\X$ reaches record levels, i.e.
$$ T^*_{n+1} \Def \inf\left\{k >T^*_n: |X_{T^*_{n+1}}| >|X_u|,\forall u <T^*_{n+1}\right\}.$$
Let $\mc{H}_n=\sigma(X_0,\dots, X_n)$.  On the event $T^*_n <\infty$, define the stopping time $t_n$ as the first time after $T^*_n$ when the process returns to the parent of  $X_{T^*_n}$.  Then 
\begin{equation}
\P(t_n=\infty|\mc{H}_{T^*_n})=\gamma, \textrm{ almost surely on }\{T_n^*<\infty\}.\label{mm2}
\end{equation}
Let $A^*=\cap_{n\in \N}\{T_n^*<\infty\} $ and $B^*=\cup_{n \in \N}\{t_n=\infty\}$.  Then, almost surely (i.e.~excepting sets of probability 0) $A^*=G_\infty${, as the walk cannot stay confined to a finite subtree of an infinite tree forever}.  Moreover, almost surely, $H^c\subset A^*=B^*=G_\infty$ (with the first equality following from \eqref{mm2}), and it follows that $\P(H,G_\infty)=0$.
\hfill
\end{proof} 

{Now, let us define a coloring scheme on the vertices of $\Gcal$ that will be useful in the proof of Theorem \ref{thm:speed}(1).}  Let $n^*$ be an integer (to be fixed later), so that only vertices lying at a level $kn^*$ for some $k \in \Z_+$ can be colored green. Firstly,  $\r$ is coloured green.   Consider a vertex $\mu$ at level $kn^*$ for $k \in \N$. Denote by $\nu$ its ancestor at level $(k-1)n^*$. Then $\mu$ is colored green if both of the following conditions hold: 
\begin{itemize}
\item $\nu$ is green, and 
\item the extension of $\X$ on $[\parent{\nu}, \mu]$ hits $\mu$ before it returns to  $\parent{\nu}$,
\end{itemize}
{where the extension of a walk is defined in Section \ref{sec:coupling1}.}
If $\nu$ is green, we call its {\it green offspring} the green vertices which are its descendants at level $|\nu| + n^*$. 
We denote $\Tcal^*$ the tree consisting of the green vertices.\\
The following result   plays a major role in our argument.
\begin{LEM}\label{extralem}
The tree $\Tcal^*$ is a {\it super-critical} Galton-Watson tree, as soon as $n^*$ is large enough.
\end{LEM}
\begin{proof}
Under the annealed measure $\P$, each green vertex has an average number of green offspring equal to  $d^{n^*} \psi_{n^*}$.
In fact, each descendant of $\nu$ at level $|\nu| + n^*$ is green with probability $\psi_{n^*}$ (we are conditioning on  $\nu$ being green!).
By linearity of expectation, we get that the (annealed) expected number of green offspring is $d^{n^*} \psi_{n^*}$.

Under $\P$, the green vertices evolve as a  Galton--Watson tree $\Tcal^*$ [see \cite{Co06} and \cite{BA17}].
  The independence of green offspring of different green vertices is due to the fact that $\mc{G}$ is a Galton-Watson tree under $\P$, and that different offspring distributions are determined by disjoint collections of independent exponential random variables (see the introduction to Section \ref{sec:speed}).  Moreover, the offspring distribution is identically distributed for each green vertex, since $\omega=\omega^*$.\\
Finally, from Corollary  \ref{cor:path}, if $du_0>1-u_1+u_0$ then 
\begin{equation}
\label{green1}
d^{n^*}\psi_{n^*}>1,
\end{equation}
as soon as we choose $n^*$ large enough, which we assume from now on.
\hfill
\end{proof}

We are now ready to prove the first part of Theorem \ref{thm:speed}.
\begin{proof}[Proof the Theorem \ref{thm:speed}(1)]
Let us fix $n^*$ large enough so that, by Lemma \ref{extralem}, the Galton-Watson tree $\Tcal^*$ is super-critical. Next we reason by contradiction. Suppose that $\P(T(\parent{\r}) <\infty)=1$.
As the extensions on each path coincide with the process  observed on that path (if $\X$ ever hits that path), until it leaves the path for good, we have that  all the green vertices at level $n^*$ will be, by definition, visited before time $T(\parent{\r})$. The same is true for the green vertices at level $2n^*$. Recursively, we can conclude that all the green vertices are visited before time $T(\parent{\r})$.
As $\P(\Tcal^* \textrm{ is infinite})>0$, we have that $\P(T(\parent{\r}) = \infty) >0$, leading to a contradiction.  This shows that in fact $\P(T(\parent{\r}) <\infty)<1$, or equivalently $\P(T(\parent{\r}) =\infty)>0$.
Together with Proposition \ref{propsansnom}, this proves Theorem \ref{thm:speed}$(1)$.
\hfill
\end{proof}

We turn to the proof of the second part of Theorem \ref{thm:speed}.  
\begin{proof}[Proof of Theorem \ref{thm:speed}$(2)$]
By Corollary \ref{cor:path}, 
$ d^n \psi_n $ decreases exponentially fast to 0.  Fix $\mc{G}$ and $\nu\in\mc{G}$, and recall that
$\psi_n=\bP^{[\parent{\r}, \nu]}(S_\nu<S_{\parent{\r}})=\bP^{[\parent{\r}, \nu]}(\X[\parent{\r}, \nu]  \mbox{ hits } \nu \mbox{ before returning to } \parent{\r})$, and $|\nu|=n$.  
By  construction of the extension $\X[\parent{\r}, \nu]$, we have that if  $\X $   hits  $\nu$ before returning to $\parent{\r}$, then the same happens for  $\X[\parent{\r}, \n]$.
Thus,
$$
\bP^{\mc{G}}(T(\nu) < T(\parent{\r})) \le\bP^{[\parent{\r}, \nu]}(S_\nu<S_{\parent{\r}})=\psi_n.
$$ 
Let $R_{\r^{-1}}=\{X_0,\dots,X_{T(\r^{-1})-1}\}$ denote the range of the walk before returning to $\r^{-1}$.  Then 
\begin{align}
 \bE^{\mc{G}}[|R_{\r^{-1}}|]&=\sum_{\nu \in \mc{G}} \bP^{\mc{G}}(T(\nu) < T(\parent{\r}))\\
&=\sum_{n=0}^\infty \sum_{\nu\in \mc{G}:|\nu|=n}\bP^{\mc{G}}(T(\nu) < T(\parent{\r}))\\
&\le \sum_{n=0}^\infty \sum_{\nu\in \mc{G}:|\nu|=n} \psi_n.
\end{align}
Taking the expectation of both sides with respect to $\P$ we get
\[\E[|R_{\r^{-1}}|]\le \sum_{n=0}^\infty \psi_n \E\left[\Big|\{\nu\in \mc{G}:|\nu|=n\}\Big|\right]=\sum_{n=0}^\infty \psi_n d^n<\infty.\]
Thus, under $\P$, the expected number of vertices visited before returning to $\r^{-1}$ is finite, which implies that the number of vertices visited before returning to $\r^{-1}$ is $\P$-almost surely finite.   This proves the claim (2). 
\hfill
\end{proof}

\subsection{Positive speed when transient}\label{sec;posspeedtrans}

We want to prove the last statement of Theorem \ref{thm:speed}: on Galton-Watson trees without leaves, the speed of the MAD walk is positive whenever $du_0>1-u_1+u_0$.   {\it We henceforth assume this condition, as well as \eqref{no_leaves} and $u_1\ge u_0$, and therefore (by Theorem \ref{thm:speed}(1)) that the random walk is transient.}

We break the argument into two pieces. First, in Section \ref{speedoutward}, we will prove the result when the underlying bias on a single path is directed outward on the tree (i.e.~$u_1\ge 1$). Second, in Section \ref{speedroot}, we will prove the positivity of the speed in the case when the underlying bias on a single path is directed towards the root of the tree (i.e.~$u_1<1$). 

The proof in Section \ref{speedoutward} relies on the presence of geometrically growing deterministic subtrees of the Galton-Watson tree. We prove that, starting from any vertex $\nu$, the walker can quite quickly go and hit the root of such a subtree. Once there, the walker has a lower-bounded probability to escape to infinity. This will prove that the walker visits $\nu$ only a small number of times before escaping to infinity. Together with the existence of a regenerative structure and the fact that, on each generation of the Galton-Watson tree, the walker only visits a geometric number of vertices, this will prove the positivity of the speed. In particular  this provides a new smooth proof of the positivity of speed in the case of regular trees with parameters $u_1=1$ and $u_0<1$, already proved  in the seminal paper \cite{DKL02}: to obtain this result on regular trees, most of the steps of Section \ref{sec:supsubtree} and Section \ref{speedoutward} can be skipped as they are dealing with the randomness of the tree.

In Section \ref{speedroot}, we deal with the bias directed towards the root. Here, we use again a regenerative structure and the fact that we can work with the annealed measure conditioned on the event that the walker does not go back to the parent of the root.  We conclude noting that, when $u_1<1$, the walker cannot visit any vertex a large number of times without touching the root.

\subsubsection{Existence of the speed}
Here we use classical arguments to prove the existence of the speed, via the regeneration levels and times of the (a.s.~transient) walk $\X$. 
\begin{DEF} \label{defregene} 
Set $\tau_0=1$ (this is the first hitting time of $\r$).  For $m\in  \N$ define recursively,
$$\tau_m=\inf\left\{k> \tau_{m-1}: \sup_{j<k}|X_j|< |X_k|\le \inf_{j\ge k}|X_j|\right\}.$$
For each $m\in \Z_+$ let $\ell_m=|X_{\tau_m}|$.
\end{DEF}

Under the measure $\P$, the sequences $((\ell_k - \ell_{k-1},\tau_k- \tau_{k-1}))_{k\ge1}$ are independent and, except for the first one, distributed like $(\ell_1,\tau_1)$ under $\P\left(\cdot|T(\parent{\r}) = \infty\right)$.  Then, using arguments of Zerner (see Lemma 3.2.5 in \cite{Zeitouni}), we can prove the following result.
%
 \begin{PRP}[The speed exists]
\label{prp:cutl}  We have $\bbE[\ell_2 - \ell_1]= \left(\bbP\left(T(\parent{\r}) = \infty\right)\right)^{-1}<\infty$. Moreover, $|X_n|/n$ converges $\P$-a.s.~to a constant $v\in [0,1]$. 
 \end{PRP}
\begin{proofsect}{Proof} For each fixed $i \in \N$,  we have 
$$ \bbP(\exists k \colon \ell_k = i) = \bbP\left(T(\parent{\r}) = \infty\right).$$
This last equality is not true in general for all Galton-Watson trees, as the probability that the tree survives up to generation $i$ depends on $i$. But here, as stated at the beginning of Section \ref{sec;posspeedtrans}, we consider Galton-Watson trees without leaves and thus we only use the fact that the walker sees a fresh environment in front of him every time it reaches a level $i$ for the first time.\\
On the other hand, 
$$
\begin{aligned}
\lim_{i \ti} \bbP(\exists k \colon \ell_k = i)  &= \lim_{i \ti} \bbP(\exists k \ge 2 \colon \ell_k = i) \\
&= \lim_{i \ti} \sum_{s=1}^\infty \bbP(\exists k \ge 2 \colon \ell_k - \ell_1 = i-s) \bbP(\ell_1=s).
\end{aligned}
$$

Using the renewal theorem, we have that 
$$ 
\lim_{i \ti} \bbP(\exists k \ge 2 \colon \ell_k - \ell_1 = i-s)  = \frac 1{\bbE[\ell_2 - \ell_1]}.
$$
The Dominated Convergence Theorem now verifies the first claim, since
$$ \bbP\left(T(\parent{\r}) = \infty\right) = \sum_{s=1}^\infty \bbP(\ell_1=s)\frac 1{\bbE[\ell_2 - \ell_1]} = \frac 1{\bbE[\ell_2 - \ell_1]}.
$$ 

For each $n$, there exists exactly one $k_n$ such that 
$\tau_{k_n-1} \le n < \tau_{k_n}$. Hence, by the definition of $(\ell_k)_{k \ge 1}$, we have 
\begin{equation}
 \frac{\ell_{k_n-1}}{\tau_{k_n}} \le \frac{|X_n|}n < \frac{\ell_{k_n}}{\tau_{k_n-1}}.\label{2sides}
\end{equation}

Note that $\ell_k=\sum_{i=1}^{k}(\ell_i-\ell_{i-1})$ and $\tau_k=\sum_{i=1}^{k}(\tau_i-\tau_{i-1})$. 
Hence, if $\bbE[\tau_2 - \tau_1]<\infty$, the law of large numbers applied to both sides of \eqref{2sides} shows that $|X_n|/n$ converges to $v=\E[\ell_2-\ell_1]/\E[\tau_2-\tau_1]$. If $\bbE[\tau_2 - \tau_1]=\infty$, then $\tau_k/k\to\infty$ a.s.~and the right-hand side of \eqref{2sides}, together with the first part of the proof, shows that $|X_n|/n\to0$.
\hfill \qed
\end{proofsect}

\subsubsection{A supercritical sub-tree} \label{sec:supsubtree}

To prove the positivity of the speed, we will use the fact that any Galton-Watson tree with no leaves almost surely contains a supercritical backbone with a regular growth (see Lemma~\ref{lemfrompem} below) . For this purpose, we need the following definition.

\begin{DEF}\label{def:Kdinf2}
Fix two integers $K\ge1$ and ${\bf d}>1$. A tree $\Gcal^*$ rooted at some vertex $\r$ is said to be $(K,{\bf d})$-infinite if there exists a subtree $\Gp$ rooted at $\r$ 
such that any vertex $\nu\in\Gp$, with $(|\nu|-|\r|)\in K\Z_+=\{Kn:n\in \Z_+\}$, has exactly ${\bf d}$ descendants at level $|\nu|+K$.
\end{DEF}

For the rest of this section, we fix two integers $K\ge1$ and ${\bf d}>1$. Let  $\Gcal^*$ be a $(K,{\bf d})$-infinite tree with root $\r$ and denote by $\Gcal' =(V', E')$ a  subtree of $\Gcal^*$ rooted at $\r$  with no leaves, and where each vertex $\nu\in \Gp$, $(|\nu|-|\r|)\in K\Z_+$, has {\it exactly} ${\bf d}$ descendants at level $|\nu|+K$.\\
Moreover, we  augment $\Gcal'$ by adding a parent $\r^{-1}$ to $\r$.
Define the Green branching process $\mc{T}:=\mc{T}(n^*)$ on $\mc{G}'$ as in Section \ref{sec:green}, with $n^*=k^*K$ being a (large enough) multiple of $K$ (under the assumption that ${\bf d}^n\psi_{nK}>1$ for all $n$ sufficiently large). Notice that $\mc{T}$ is {\it not} a homogeneous Galton--Watson tree due to the possibly irregular structure of $\Gcal'$. \\

\begin{DEF}  \label{def:Kdinf}
Let 
$ \beta_G(n^*,\omega)$ denote the survival probability of the green branching process $\Tcal$, with $n^*=k^*K$, i.e.~
$$\beta_G(n^*,\omega)=\bP^{\Gcal'}_\omega(|\Tcal(n^*)|=\infty).$$
\end{DEF}
We will  prove next that $\mc{T}$ is infinite with positive probability. We prove this when the initial configuration is $\omega^*$, using  large deviations arguments.
\begin{LEM} \label{lemspb}
If ${\bf d}^n\psi_{nK}>1$ for all $n$ sufficiently large then there exists $k^*$ such that $\beta_G(n^*,\omega^*) > 0$.
\end{LEM}
\begin{proof}
Let $k^*$ be such that $r^*:={\bf d}^{k^*}\psi_{k^*K}>1$, and let $n^*=k^*K$.  Suppose that $\nu$ is green. The distribution of green descendants of  $\nu$ at level  $|\nu| + n^*$  is determined by the finite subtree consisting of  $\nu$, its parent and its descendants up to level $|\nu| + n^*$.  
We say that $\nu$ is the root of this $n^*$ subtree.  Note that there are only $M<\infty$ (i.e.~finitely many) distinct trees without leaves that have exactly ${\bf d}^{k^*}$ descendants at level $n^*$, so we can label them as $1, 2, \ldots, M$.   Hence, the process describing the green vertices in this context evolves as a peculiar non-homogeneous branching process with at most $M$ different offspring distributions $\xi_1, \xi_2, \ldots, \xi_M$, each of which: has mean $r^*>1$ (so is supercritical) and is bounded by $L:={\bf d}^{k^*}$.
We were not able to find a precise statement in the literature that would imply the supercriticality of this branching process, so we include a proof.


Denote by $Z_n$ the set of vertices at level $n$ which are colored green. Hence $Z_n$ is a random subset of level $n$, possibly empty. Fix $r  \in (1,  r^*)$.
Set $D_n = \{|Z_n| \ge r ^n\}$.
Next we prove that  there exist  $c,\gamma>0$, depending only upon the choice of $r $, such that
\begin{equation}\label{eq:branch}
\bP^{\mc{G}'}\left(D_n \;\Big|\; \bigcap_{k=1}^{n-1} D_k\right) \ge 1 - c{\rm e}^{- \gamma \floor{r ^n}},
\end{equation}
with the left hand side always being strictly positive.  The latter statement is trivial.  We  will prove the former via Azuma-Hoeffding inequality.

Fix a sequence $\bl:= (\lambda_k)_{k \in \N}$, with each $\lambda_k\in \{\xi_1, \xi_2, \ldots, \xi_M\}$, and consider a sequence of independent random variables $(U_k)_{k\in \N}$, with marginal laws $(\lambda_k)_{k \in \N}$.  Let $\bP_{\bl}$ denote the joint law of the independent sequence $(U_k)_{k\in \N}$.
We claim that there exist $\gamma',c>0$ such that for some ${\bl}$ (depending on $Z_{n-1}$), almost surely
\begin{equation}\label{eq:inb}
\bP^{\mc{G}'}\left(D_n \;\Big|\;\bigcap_{k=1}^{n-1} D_k, Z_{n-1}\right) \ge \bP_{\bl}\Big(\sum_{k=1}^{\floor{r ^{n-1}}} U_k >   r ^{n}\Big)\ge 1-c{\rm e}^{-\gamma' \floor{r ^{n-1}}}.
\end{equation}
Since the right hand side does not depend on $Z_{n-1}$ (or $\bl$) this estimate proves \eqref{eq:branch}.

The first inequality of \eqref{eq:inb} comes from the fact that on $D_{n-1}$ we have that there are at least $\floor{r ^{n-1}}$ green vertices, and $Z_{n-1}$ provides the information on their location, which gives us information of the \lq type\rq\/ of each offspring distribution (in particular this tells us $\lambda_1,\dots, \lambda_{\floor{r ^{n-1}}}$). Hence, in order to understand the size of the next generation, by reducing the number of green vertices at level $n-1$ to exactly $\floor{r ^{n-1}}$ we stochastically reduce the growth, proving the first inequality of \eqref{eq:inb}.

For the second inequality of \eqref{eq:inb}, using Azuma-Hoeffding inequality on the centred {and bounded} variables (for $n$ sufficiently large),
\begin{align}
\bP_{\bl}\Big(\sum_{k=1}^{\floor{r ^{n-1}}} U_k \le   r ^{n}\Big)&=\bP_{\bl}\Big(\sum_{k=1}^{\floor{r ^{n-1}}} (U_k-r^*) \le   r ^{n}-\floor{r ^{n-1}}r^*\Big)\\
&\le \exp\left\{-\frac{(r ^{n}-\floor{r ^{n-1}}r^*)^2}{2\floor{r ^{n-1}}L^2}\right\}.
\end{align}
The argument of the exponential is 
\begin{align}
-\frac{\floor{r ^{n-1}}}{2L^2}\left(r \frac{r^{n-1}}{\floor{r ^{n-1}}}-r^*\right)^2.
\end{align}
For $\vep\in (0,r^*-r)$ and all $n$ sufficiently large this is at most 
$$
-\frac{\floor{r ^{n-1}}}{2L^2}(r(1+\vep/2)-r^*)^2,
$$
which proves the second inequality of \eqref{eq:inb} for sufficiently large $n$, and the constant $c$ makes the statement true for all $n$.

Finally, notice that the event that there are infinitely many green vertices contains the event  $\bigcap_{i=1}^\infty D_i$, and we have 
 $$ \bP^{\mc{G}'}\Big(\bigcap_{n=1}^\infty D_n\Big)= \prod_{\ell=1}^\infty \bP^{\mc{G}'}(D_\ell|\cap_{i=1}^{\ell-1}D_i)>0,$$
 proving the lemma.
\hfill
\end{proof}
For the remaining part of this subsection, we  consider the behaviour of a MAD walk $\bX$ defined on $\Gcal'$, for arbitrary configuration $\omega\in \Ccal(\Gcal')$, starting at  $X_0=\parent{\r}$.  
Again we can define the Green branching process $\mc{T}(n^*,\omega)$ in exactly the same way, with $n^*=kK$ being a (large enough) multiple of $K$. Because $\omega$ differs from $\omega^*$ on finitely many edges, the probability that $\mc{T}(n^*,\omega)$ survives forever is also positive.
A MAD $\X$ on $\mc{G}'$ satisfying ${\bf d}^n\psi_{nK}>1$ for $n$ sufficiently large is said to be {\it supercritical}.

\begin{DEF} Consider a MAD walk on $\Gcal'$.  For any  $\omega \in \Ccal(\Gcal')$ define 
$$ \beta(\omega) \Def \bP^{\Gcal'}_{\omega}\left(T(\parent{\r}) = \infty\right).$$
\end{DEF}


For each $\omega \in \Ccal(\Gcal')$, since $\{|\mc{T}(n^*)|=\infty\}\subset \{T(\r^{-1})=\infty\}$ on our probability space, it follows immediately that
\begin{equation}\label{betain}
  \beta(\omega) \ge \beta_G(n^*,\omega) > 0.
  \end{equation}

\begin{LEM}\label{betaa} Suppose that $\omega_1 \ge \omega_2$ (recall Definition~\ref{DEF:2.5}).  Then 
\[ \beta(\omega_1) \ge \beta_G(n^*,\omega_1)  \ge \beta_G(n^*,\omega_{2}) > 0.\]
\end{LEM}
\begin{proof} The first inequality is exactly \eqref{betain}  for $\omega_1$. The middle inequality 
is immediate from Lemma \ref{lem:pathmono}.  The last inequality  is trivial from the supercriticality of $\Tcal$.  
\hfill
\end{proof}


\subsubsection{Positivity of the speed for outward bias: the case $u_1 \ge1$} \label{speedoutward}

Let $T_n=\inf\{k\ge 0:|X_k|=n\}$ denote the hitting time of level $n$.  Our main tool in the proof of positivity of the speed when $u_1\ge 1$ is  the following proposition.

\begin{PRP}\label{mainth}  Under the condition that $u_1\ge1$, there exists $C\in (0,\infty)$ such that 
\begin{equation}\label{trel} 
\bbE \left[\frac{T_n}{n}\right] \le C <\infty.
 \end{equation}
\end{PRP}

%

Denote by $\xi_k$ the number of vertices visited by the process $\X$ at level $k$, i.e. 
$$ \xi_k = |\mbox{Range}(\X) \cap \{ \nu \in V \colon |\nu| =k\}|,$$
where 
$$ \mbox{Range}(\X) = \{X_k \colon k \in \Z_+\}.$$
Let $\beta^*=\P(T(\r^{-1})=\infty)$.  Each time a vertex at level $k$ is visited for the first time, with probability $\beta^*$ (independent of the past) the process never returns to level $k-1$. Hence (under $\P$) 
the random variable $\xi_k$ is stochastically dominated by a geometric random variable $\widetilde{\xi}$ with average 
$1/\beta^*$.  It follows that for all $k \in \N$ we have 
\begin{equation}\label{moment}
\bbE[\xi_k] \le  \frac{1}{\beta^*}.
\end{equation}

Fix $\mc{G}$.  For $\nu\in \mc{G}$, let  
$$ L(\nu, \infty) \Def |\{ j \ge 0 \colon X_j = \nu\}|$$ denote
the number of visits  to $\nu$ by $\X$. 
Denote by $\Gp_{\nu}=(V'_{\nu},E'_{\nu})$ the subtree of $\Gcal$ whose vertex set consists of $\parent{\nu}$, $\nu$, and all the descendants of $\nu$.  Define
$$  L'(\nu, \infty) \Def |\{ j \ge 0 \colon X^{\ssup {\Gp_{\nu}}}_j = \nu\}|,$$
i.e. the time spent by the extension of $\X$ to $\Gp_{\nu}$ in $\nu$. Notice that (see Definition~\ref{spect})
$\Gp_{\nu}$ is a special tree with $\r' = \parent{\nu}$.

For $\nu\in \mc{M}\setminus \mc{G}$ we set $L(\nu,\infty)=L'(\nu,\infty)=0$.

\begin{LEM}\label{LProp} Let $\nu\in \mc{M}$.
\begin{enumerate} 
\item[$1)$] Under $\P$, the random variable $\E[L(\r, \infty)|\mc{G}]$ has the same distribution as: $\E[L'(\nu, \infty)|\mc{G}'_\nu]$ conditional on $\nu\in \mc{G}$.
\item[$2)$] Under $\P$, conditional on $\nu\in \mc{G}$, the random variable $\E[L'(\nu, \infty)|\mc{G}'_\nu]$ is independent of $\1_{T(\nu)<\infty}$.
\item[$3)$]  For $P$-almost every $\mc{G}$,  $L(\nu, \infty) \le  L'(\nu, \infty)$ on the probability space in Section \ref{sec:coupling1}.
\end{enumerate}
\end{LEM}
\begin{proof}
These are all simple consequences that can be seen from the construction of the extension walks on the same probability space in Section \ref{sec:coupling1}.\\
 1)  Is immediate as  $\Gcal'_{\nu}$ (conditional on $\nu\in \mc{G}$) has the same distribution as $\Gcal$. \\
2) Is a consequence of the fact that (conditional on $\nu\in \mc{G}$), $\E[L'(\nu, \infty)|\mc{G}'_\nu]$ is a function of $\mc{G}'_\nu$, while the event $T(\nu)<\infty$ is independent of $\mc{G}'_{\nu}$.

\noindent 3)  For fixed $\mc{G}$, on the probability space of Section \ref{sec:coupling1}, the moves of the two processes  $\X^{\ssup{\Gp_{\nu}}}$ and $\X$ on $\Gp_{\nu}$ coincide up to the (possibly infinite) time when the latter leaves $\Gp_\nu$ forever. Hence, the process  $\X^{(\Gp_\nu)}$ visits $\nu$ at least the number of times that  $\X$ does on this probability space. \hfill
\end{proof}

\begin{LEM}\label{loct} If {$u_1\ge u_0$ and} $u_1\ge 1$ then, for all $\omega \in \Ccal$, there exists $C''< \infty$ such that 
\begin{equation}\label{visit}
  \bbE_\omega[L'(\nu, \infty)] <C''.
  \end{equation}
  \end{LEM}
  Recall the Definition \ref{def:Kdinf2} of a $(K,{\bf d})$-infinite tree.
To prove Lemma \ref{loct} we need the following result of Pemantle \cite{Pemtree}. 

\begin{LEM}\label{lemfrompem}[Lemma 6 of \cite{Pemtree}]
Let $\mc{G}$ be a Galton-Watson tree with no leaves and with mean offspring $d>1$. Fix  $m \in (1, d)$. There exists an integer $K_0\ge1$ such that, for any $K\ge K_0$,
\[
\bbP\left(\mc{G} \text{ is }(K,\lfloor m^K\rfloor)\text{-infinite}\right)\ge \frac{1}{2}.
\]
\end{LEM}
\begin{REM}
Pemantle  \cite{Pemtree} states that there exists an integer $K$ such that the conclusion holds but a quick inspection of the proof reveals the former result.\\
In  \cite{Pemtree}, a branching process is said to be ${\bf d}$-infinite if it is $(1,{\bf d})$-infinite according to Definition \ref{def:Kdinf2}. Moreover, in \cite{Pemtree}, for a branching process $B$, $B^{(K)}$ denotes the branching process consisting of the particles of $B$ in generations $K\Z_+$, and therefore $B$ is  $(K,\lfloor m^K\rfloor)$-infinite if and only if $B^{(K)}$ { is } $\lfloor m^K\rfloor${-infinite}.
\end{REM}

A consequence of this last result is the following.
\begin{LEM}\label{Pem-ref}
Let $\mc{G}$ be a Galton-Watson tree with no leaves with mean offspring $d>1$. Fix $m\in (1, d)$.  There exists an integer $K_0\ge1$ such that, for any $K \ge K_0$, 
\[
\bbP\left(\exists \nu\in \mc{G} \colon \Gp_\nu \text{ is } (K,\lfloor m^K\rfloor)\text{-infinite}\right)=1.
\]
\end{LEM}
\begin{proof}
Let $(\eta_i)_{i \in \N}$ be an i.i.d.~sequence of r.v.~with marginal distribution equal to the offspring distribution of $\Gcal$, and let $p=\P(\eta_0\ge2)>0$. Then a way to generate the Galton-Watson tree is the following procedure. First, fix an infinite ray, or half-line, composed of vertices $\r \sim \nu^{\ssup 1} \sim \nu^{\ssup 2}\ldots$, with $|\nu^{\ssup i}| =i$.  Then for any $j\ge1$, attach $\eta_j-1$ new children to $\nu^{\ssup j}$ and from each of them, start an independent Galton-Watson tree distributed like $\Gcal{\tiny }$.
Define 
$$N=\inf\{j\ge 1\colon \ \nu^{\ssup j}\text{ has a child } \nu'\neq\nu^{\ssup {j+1}}\text { s.t. } \Gcal_{\nu'}'\text{ is } (K,\lfloor m^K\rfloor)\text{-infinite}\}.$$
Note that $N$ is independent of $\del({\r})$.
Moreover, the random variable $N$ is dominated by a geometric random variable with parameter $p/2$, as this corresponds to the probability that a given vertex has at least two children and that one given child grows an infinite regular backbone. Hence, $N$ is finite almost surely, which proves the Lemma. \hfill
\end{proof}

\begin{proof}[Proof of Lemma \ref{loct}]
For any $\omega \in \Ccal$, the fact that $u_1\ge u_0$ together with Lemma \ref{lem:pathmono} implies that
\[
\bbE_{\omega^*}[L'(\nu,\infty)]\ge \bbE_{\omega}[L'(\nu,\infty)].
\]
Hence, it is sufficient to {prove the statement for} $\omega=\omega^*$.  In virtue of Lemma~\ref{LProp} 1)  it is enough to prove  that  there exists a constant $C''$ such that 
\begin{equation}\label{visitr}
\bbE[L(\r, \infty)] <C''.
\end{equation}
Since $du_0>1-u_1+u_0$, we can find an $m\in(1,d)$ such that $mu_0>1-u_1-u_0$ and hence $m^n\psi_n>1$ for all $n$ sufficiently large.  Moreover we can find such an $m$ and a $K_0$ (being provided by Lemma \ref{lemfrompem}) such that ${\bf d}^n\psi_{nK}>1$ for all $K\ge K_0$,  where ${\bf d} =  \lfloor m^K\rfloor$.
Let $\r \sim \nu^{\ssup 1} \sim \nu^{\ssup 2}\sim\ldots $ and $N$ as in the proof of Lemma~\ref{Pem-ref} and let $\bar{\nu}$ be the children of $\nu^{\ssup N}\neq \nu^{\ssup {N+1}}$ such that $\Gp_{\bar{\nu}}$ is $(K,\lfloor m^K\rfloor)$-infinite.
Moreover, recall that $\P(N \ge n) \le (1-p/2)^n$ for some $p>0$. Hence
\begin{equation}\label{lawoftotexp}
 \bbE[L(\r, \infty)]  \le  \sum_{n=1}^\infty \bbE[L(\r, \infty)\;|\; N=n] (1-p/2)^n.
\end{equation}
Next, we prove that 
\begin{equation}\label{loctimeatnu}
\bbE[L(\r, \infty)\;|\; N=n] \qquad \mbox{ grows polynomially in } n,
\end{equation}
 which combined with \eqref{lawoftotexp} would end the proof.
Let us prove that, on $\{N=n\}$, $L(\r, \infty)$ is stochastically dominated by a non-degenerate geometric random variable. For this purpose, we bound the probability that the walker jumps from $\r$ to $\nu^{\ssup 1}$, then hits $\bar{\nu}$ before $\r$ and then escapes to infinity from $\bar{\nu}$ without hitting $\nu^{\ssup N}$ (and thus $\r$).

First, since $u_0\le u_1$, the probability for $\mathbf{X}$ to jump from $\r$ to $\nu^{\ssup 1}$, given $\del(\r)$, is at least $u_0/((1+\del(\r))u_1)$.\\
Second, each time $\mathbf{X}[\r, \bar{\nu}]$, the extension of $\mathbf{X}$ to the path $[\r, \bar{\nu}]$, is at $\nu^{\ssup 1}$ it has a probability  at least $\psi_{n}$ to hit $\bar{\nu}$ before $\r$.\\
Third, as $\bar{\nu}$ is the root of a $(K,\lfloor m^K\rfloor)$-infinite tree, each time the extension process $\X^{\Gp_{\bar{\nu}}}$ is on $\bar{\nu}$, the probability that it never returns to  $\nu^{\ssup n}$, the parent of $\bar{\nu}$, is bounded away from zero by a positive constant $\beta^*$, due to Lemmas~\ref{lemspb} and \ref{betaa} and the fact that $u_0\le u_1$. 

All of this provides a lower-bound on the  quenched probability (given $\{N=n\}$) that $\mathbf{X}$ escapes from $\r$ to infinity, regardless of the past history.  Specifically, for any environment $\tilde{\omega}\in\mathcal{C}$, we have that
\[
\bP^{\mc{G}}_{\tilde{\omega},\r}\left[T(\r^{-1})=\infty\right]\ge\frac{u_0}{(1+\del(\r))u_1}\times \psi_n\times \beta^*.
\]
Hence, under the quenched measure, the number of visits to $\r$ is bounded by a geometric random variable and we obtain that, on $\{N=n\}$,
\[
\bE^{\mc{G}}[L(\r, \infty)]\le C\del(\r)\psi_n^{-1}.
\]
Note that $\psi_n$ is deterministic and that the random variables $\del(\r)$ and $N$ are independent, hence $\bbE[L(\r, \infty)\;|\; N=n] \le Cd\psi_n^{-1}$.  
Finally, by Corollary \ref{cor:path}, $ (\psi_n)^{-1}$ is of the order of $n^{u_0^{-1}}$ in the case $u_1=1$ and of the order of $1$ when $u_1>1$.
This implies the result.
\hfill
\end{proof}

\begin{proofsect}{Proof of Proposition~\ref{mainth}} Recall  that $T_n$ is the hitting time of level $n$ by $\X$. 

Utilizing Tonelli's Theorem for non-negative functions we have
\begin{equation}
\begin{aligned}
\bbE[T_n] &\le  \bbE\Big[\sum_{j=0}^{n-1} \sum_{\nu\in \mc{M} \colon |\nu| = j } L'(\nu, \infty)\1_{T(\nu) < \infty}\Big]  \qquad \Big(\mbox{by Lemma~\ref{LProp},  3)}\Big)\\
&= \sum_{j=0}^{n-1} \sum_{\nu \colon |\nu|= j}   \bbE \Big[ L'(\nu, \infty)\Big] \bbP(T(\nu) < \infty)\qquad \Big(\mbox{by Lemma~\ref{LProp},  2)}\Big)\\
&\le C'\sum_{j=0}^{n-1} \sum_{\nu \colon |\nu|= j}  \bbP(T(\nu) < \infty) \qquad \Big(\mbox{by Lemma~\ref{loct}}\Big)\\ \label{lasteqtn}
&\le C'\sum_{j=0}^{n-1}  \bbE[\xi_j] \le Cn \qquad \Big(\mbox{by \eqref{moment}}\Big).
\end{aligned}
\end{equation}
\hfill \qed

\end{proofsect}

\medskip

\begin{proofsect}{Proof of positivity of the speed in Theorem \ref{thm:speed}, in the case $u_1\ge1$}
As  the speed, i.e. $v = \lim_{n \ti} |X_n|/n$
    exists in $[0, \infty]$, $\bbP$-a.s.,  we have that $n/T_n\equiv X_{T_n}/T_n\ra v$ almost surely.  Thus $\lim_{n \ti} T_n/n = 1/v$, $\bbP$-a.s.~(where we interpret $1/v=+\infty$ if $v=0$). Using Fatou's lemma and Proposition \ref{mainth} we have
\begin{equation} \label{concpositive}
  \frac 1v = \lim_{n \ti} \frac {T_n}n = \bbE\left[\lim_{n \ti} \frac {T_n}n\right] \le \liminf_{n \ti} \bbE\left[\frac {T_n}n\right] <\infty.
  \end{equation}
\hfill \qed
\end{proofsect}

%
%
%
%

\subsubsection{Bias towards the root: the case $u_1<1$}\label{speedroot}

In this section, we prove positivity of the speed when $u_1<1$.  
Recall the regenerative structure introduced in Definition \ref{defregene} and Proposition \ref{prp:cutl}. In particular, using only the fact that $\bbP\left(T(\parent{\r}) = \infty\right)>0$, we know that $\bbP$-a.s.~the speed exists, i.e.~$|X_n|/n$ converges. Moreover,
\[
v=\lim_{n\to\infty}\frac{|X_n|}{n}=\lim_{n\to\infty}\frac{|X_{\tau_1+n}|-|X_{\tau_1}|}{n},\ \bbP\text{-a.s.}
\]
and $|X_{\tau_1+n}|-|X_{\tau_1}|$ under $\bbP$ has the same law as $|X_n|$ under the conditional measure $\overline{\bbP}(\cdot):=\bbP\left(\cdot|T(\parent{\r}) = \infty\right)$. Again, we want to estimate the number of returns to a vertex and bound $\overline{\E}[T_n]$.

Under $\P$, on the event that $T(\nu)<\infty$, the offspring of $\nu$ has the usual offspring distribution of a Galton-Watson process with law $P$, and we can consider the extension ${\bf X}^{(\nu)}$ of the walk ${\bf X}$ on the (finite) subtree composed by the unique self-avoiding paths between $\parent{\r}$ and the children of $\nu$.
Let $L^{(\nu)}$ denote the number of visits to $\nu$ by ${\bf X}^{(\nu)}$ between its first hitting time of $\nu$ and its subsequent hitting  time of $\parent{\r}$.  Then 
$$L(\nu,\infty) \1_{T(\nu)<\infty} \1_{T(\parent{\r}) = \infty}\le \1_{T(\nu)<\infty} L^{(\nu)},$$ 
where the left hand side is the number of visits of  $\mathbf{X}$ to $\nu$ on the event that $\nu$ is hit but the walk never returns to $\parent{\r}$.

When ${\bf X}^{(\nu)}$ first reaches $\nu$, the environment seen at this time is the environment $\omega_\nu$ such that $\Gamma_{\omega_\nu}$ contains every edge from $\r^{-1}$ to $\nu$ and no others.  Recall that $\P_{\omega_\nu,\nu}$ is the annealed law of a walk started at $\nu$ in environment $\omega_\nu$ (conditional on $\nu \in \mc{G}$).  Then 
\[\E[\1_{T(\nu)<\infty} L^{(\nu)}]=\E\left[\1_{T(\nu)<\infty} \E_{\omega_\nu,\nu}[L^{(\nu)}]\right]=\P(T(\nu)<\infty) \E_{\omega_\nu,\nu}[L^{(\nu)}].\]

Under the measure $\P_{\omega_\nu,\nu}$, in order for ${\bf X}^{(\nu)}$ to hit the root before coming back to $\nu$, it needs to step to $\nu^{-1}$, and then (on a one-dimensional segment) hit the root before $\nu$.  Under this measure, after first stepping to $\nu^{-1}$, ${\bf X}^{(\nu)}$ has the law of a random walk with drift on a segment of $\Z$.  Since $u_1<1$, this random walk has a negative drift, with left step probability $p(u_1)=1/(1+u_1)>1/2$, and thus it has probability at least $h(u_1)=(2p(u_1)-1)/p(u_1)=1-u_1>0$ of hitting $\r^{-1}$ before returning to $\nu$, irrespective of $|\nu|$.

Thus, regardless of the distance between $\nu$ and $\r$, 
\begin{eqnarray}\nonumber
\E_{\omega_\nu,\nu}[L^{(\nu)}]&=& \sum_{k\ge0} \P_{\omega_\nu,\nu}\left(L^{(\nu)}\ge k+1\right)\\
&=&\sum_{k\ge0} \E_{\omega_\nu,\nu}\left[\P_{\omega_\nu,\nu}\left(L^{(\nu)}\ge k+1\big|\del(\nu)\right)\right]\\
&\le & \sum_{k\ge0} \E_{\omega_\nu,\nu}\left[\left(1-\frac{1}{u_1  \del(\nu)+1}\times h(u_1)\right)^k\right].
\end{eqnarray}
By moving the sum inside the expectation and summing we obtain
\begin{eqnarray}\nonumber
\E_{\omega_\nu,\nu}[L^{(\nu)}]&\le &\E_{\omega_\nu,\nu}\left[\frac{u_1  \del(\nu)+1}{1-{u_1} }\right]= \frac{u_1  d+1}{1-{u_1} }.  
\label{silverlining}
\end{eqnarray}
It follows that
\begin{align}
\overline{\E}\left[\1_{T(\nu)<\infty}L(\nu,\infty)\right]&=\frac{\E\left[\1_{T(\nu)<\infty}\1_{T(\parent{\r}) = \infty}L(\nu,\infty)\right]}{\P(T(\parent{\r}) = \infty)}\\
&\le \frac{\P(T(\nu)<\infty) \E_{\omega_\nu,\nu}[L^{(\nu)}]}{\P(T(\parent{\r}) = \infty)}\\
&\le \frac{\P(T(\nu)<\infty)}{\P(T(\parent{\r}) = \infty)}\times \frac{u_1  d+1}{1-{u_1} }.
\end{align}

\blank{
In particular, we have that:
\begin{eqnarray*}
\overline{\bbE}\left[\1_{T(\nu)<\infty}L(\nu,\infty)\right]&\le& \frac{{\bbE}\left[ \bP^{\mc{G}}\left(T(\nu)<\infty)\right) E_\nu[L^{(\nu)}]\right]}{\bbP\left[T(\parent{\r}) = \infty\right]}\\
&\le& \frac{\bbP\left[T(\nu)<\infty\right] }{\bbP\left[T(\parent{\r}) = \infty\right]}\times\frac{4\bbE[d_\nu]}{1-{u_1} }\\
&\le& \frac{\bbP\left[T(\nu)<\infty\right] }{\bbP\left[T(\parent{\r}) = \infty\right]}\times\frac{4d}{1-{u_1} },
\end{eqnarray*}
where we used the Markov property, the bound \eqref{silverlining} and the fact that $\1_{T(\nu)<\infty}$ and $d_\nu$ are $P$-independent, where  $P$ is the law of the Galton-Watson tree.\\
}
But, we have that
\[\overline{\E}[T_n]=\sum_{j=0}^{n-1}\sum_{\nu\in \mc{M}:|\nu|=j}\overline{\E}\left[\1_{T(\nu)<\infty}L(\nu,\infty)\right].\]
So repeating the argument in \eqref{lasteqtn} we see that there exists a finite constant $C$ such that $\overline{\E}\left[T_n\right]<Cn$. In turn, using the same argument as in \eqref{concpositive}, one can conclude that there exists $v>0$ such that $|X_n|/n$ converges to $v$   $ \overline{\bbP}$-a.s.. As $\tau_1 <\infty$, we have that $|X_n|/n$ converges to $v$  $ \bbP$-a.s., which concludes the proof.

\begin{REM}
Note that it is crucial here to work under the conditioned measure $\overline{\bbP}$ in order to prove that a vertex cannot be too much visited. Indeed, it could happen that, before the first regeneration time, the walker visits a very large number of time the root. This behavior is particular to the first regeneration slab and cannot happen again later.
\end{REM}

\section{Monotonicity of the speed}
\label{sec:mono}
 Here, we divide the proof of Theorem \ref{thm:mono} into two main subsections. In Section \ref{sec:genstate}, we explain the idea of the proof, which is an adaptation of an argument of Ben Arous, Fribergh and Sidoravicius \cite{BAFS}. We provide general statements that one should be able to apply in a variety of situations  by simply checking some key requirements.\\
 In Section \ref{proofmult}, we apply these results to the Once-reinforced random walks in the multiplicative case. The core of the proof is a coupling, for which we give a complete construction. In Section \ref{proofadd}, we explain how to apply the same general results to the additive case, providing much less details.
 
\subsection{General statements} \label{sec:genstate}
When the underlying bias is very strong, there are at least 3 different strategies that one can try to employ to prove monotonicity of the speed for small values of the reinforcement.  
Both expansion methods (such as \cite{HH12,HH10,H12,HSun12}) and Girsanov transform methods (such as \cite{Pham1,Pham2}) have been successfully used to prove regularity properties of the velocity for highly transient (e.g.~by taking $d$ large) self-interacting walks on $\Z^d$.  It is plausible that similar techniques can be employed on trees as well, and it would be of interest to see if significant improvements to Theorem \ref{thm:mono} can be made by employing these methods.  

Coupling is perhaps the most natural approach when trying to prove monotonicity.  When $d=1$ one can prove monotonicity in $\beta>0$  for all $\alpha>1$ by coupling e.g.~as in \cite{HS12,Hol15}, but on the other hand it is an easy exercise to compute the speed directly in this setting (with $\alpha\ge 1$) as 
\[
\frac{\alpha-1}{\alpha+1+2\beta}
,\]
in the multiplicative setting and 
\[\frac{\alpha(\alpha-1)}{\alpha(\alpha+1+\beta)+2\beta(1+\beta)},\]
in the additive setting.
Coupling has also been successfully employed to prove monotonicity of the speed of random walk in random environments where the local environment can only take two different values \cite{HS14}.  

Neither of the above coupling strategies seem to be applicable here.  We therefore adapt the coupling approach of Ben Arous, Fribergh and Sidoravicius \cite{BAFS}.  Denote by $\T_d$ the regular rooted tree where each vertex has exactly $d+1$ neighbors, with the exception of $\parent{\r}$ which has exactly one neighbor, i.e. $\r$. The main idea is to couple three walks:
\begin{enumerate}
\item[(1)] $\Xb$ a ORbRW on $\T_d$ with bias $\alpha$ and with reinforcement parameter $\beta$;
\item[(2)] $\Xe$ a ORbRW on $\T_d$ with bias $\alpha$ and with reinforcement parameter $\beta+\vep$;
\item[(3)] $\Y$ a biased random walk on $\Z$,
\end{enumerate}
such that both $\X$ walks step away from the root when $\Y$ steps to the right.  This gives a common regeneration structure and consequently a useful expression for $v(\beta)-v(\beta+\vep)$.  

The main difference of our work compared to the case studied in \cite{BAFS} is in the way in which the walks can decouple.  

To state the relevant results explicitly, we recall the concept of regeneration times and levels of walks on $\Z$.  We say that $N$ is a regeneration time for a walk $\zb=(Z_n)_{n \in \Z_+}$ on $\Z$ if $\min_{m\ge N}Z_m>\max_{n<N}Z_n$.  In this case we say that $Z_{N}$ is a regeneration level for $\zb$.  Let $\mc{D}_{\zb}$ denote the set of regeneration {times} for $\zb$.

If $\Y$ is a random walk on $\Z$ with positive bias, then the set of regeneration times $\mc{D}_Y$ is almost surely infinite and with positive probability $0 \in \mc{D}_Y$.  Let $\mc{D}_0=\{0\in \mc{D}_Y\}$.  Let $\tau_1$ denote the first strictly positive regeneration time for $\Y$, and for each $k\ge 1$ recursively define $\tau_{k+1}$ to be the first regeneration time for $\Y$ after time $\tau_{k}$.

\begin{LEM}
\label{lem:super0}
Suppose that $\zb$ and $\zb'$ are nearest neighbour walks (on $\Z$ or $\Z_+$) and that $\Y$ is a nearest neighbour simple random walk on $\Z$ with $\P(Y_1=1)=p>1/2$, all on the same probability space such that:
\begin{itemize}
\item[(i)] $Z_{n+1}-Z_{n}=Z'_{n+1}-Z'_{n}=1$ whenever $Y_{n+1}-Y_n=1$,
\end{itemize}
then the regeneration times $\tau_i$ of $\Y$ are also regeneration times for the walks $\zb$ and $\zb'$. Moreover, if
\begin{itemize}
\item[(ii)] $(Z_{\tau_{i+1}}-Z_{\tau_i})_{i \ge1}$ are i.i.d.~random variables and $(Z'_{\tau_{i+1}}-Z'_{\tau_i})_{i \ge1}$ are i.i.d.~random variables, with $Z_{\tau_{i+1}}-Z_{\tau_i}$ and $Z'_{\tau_{i+1}}-Z'_{\tau_i}$ being independent of $(\tau_k:k\le i)$ for each $i$, and
\item[(iii)] $\E\left[Z_{\tau_1}-Z'_{\tau_1}\big|\mc{D}_0\right]>0$.
\end{itemize}
Then there exist $v>v'>0$ such that $\P(n^{-1}Z_n\ra v, n^{-1}Z'_n\ra v')=1$.
\end{LEM}
We will apply the above result to walks of the form $Z_n=|X_n|$, where $X_n$ is a walk on $\T_d$.

Let $\mathfrak{B}=\left\{ 0\le i<\tau_1: Y_{i+1}-Y_i=-1\right\}$ denote the set of times before $\tau_1$ when $\Y$ takes a step back.  For two random walks $\X',\X$ on $\T_d$, let $\delta(X',X)=\inf\left\{i\le\tau_1:\ X'_i\neq X_i\right\}$.  Let $\overline{\P}(\cdot)=\P(\cdot|\mc{D}_0)$.  We will prove the following result.

\begin{LEM}
\label{lem:tree_walks}
Let $\X, \X'$ be two nearest neighbour walks on $\T_d$, and $\Y$ a biased random walk on $\Z$ such that $\zb=|\X|$ and $\zb'=|\X'|$ and $\Y$ satisfy the assumptions of Lemma \ref{lem:super0}(i),(ii).  Suppose also that $\overline{\P}\big(|\mathfrak{B}|=1, |X_{\delta}|-|X'_{\delta}|<0\big)=0$ and 
\begin{align}
\overline{\P}\big(|\mathfrak{B}|=1, |X_{\tau_1}|-|X'_{\tau_1}|\ge 1\big)>\sum_{k=2}^{\infty}2k\overline{\P}(|\mathfrak{B}|=k,|X_{\tau_1}|-|X'_{\tau_1}|\ne 0).\label{trousers11}
\end{align}
Then (iii) of Lemma \ref{lem:super0} holds, and in particular $v>v'$.
\end{LEM}

The following is \cite[Lemma 4.1]{BAFS}, and states that $|\mathfrak{B}|=k \Rightarrow \tau_1\le 3k+1$, almost surely.
\begin{LEM}\label{majtau}
$\op(\{|\mathfrak{B}|=k\}\setminus \{\tau_1\le 3k+1\})=0$.
\end{LEM}
As a result of Lemma \ref{majtau} we see that
\[\op(|\mathfrak{B}|=k,|X_{\tau_1}|-|X'_{\tau_1}|\ne 0)=\op(|\mathfrak{B}|=k, \tau_1\le 3k+1,|X_{\tau_1}|-|X'_{\tau_1}|\ne0).\]
The right hand side is bounded above by the probability of $\Y$ taking at least $k$ backward steps in the second through $3k$-th steps (the first and last steps are a.s.~positive under $\op$ by definition of $\tau_1$), with at least one resulting in a decoupling.  For our coupling of reinforced random walks with reinforcement parameters $\beta'>\beta>0$ this event will have  probability of order $C(1-p)^{k} (\beta'-\beta)$, giving that the right hand side of \eqref{trousers11} is of order $(1-p)^2(\beta'-\beta)$.  On the other hand the left hand side of \eqref{trousers11} will be of the order of $(1-p)(\beta'-\beta)$.  Thus in our setting we will be able to choose a coupling satisfying the requirements of Lemma \ref{lem:tree_walks} with $p$ very close to 1.



\subsubsection{Proof of Lemma \ref{lem:super0}} 
We first show that (i) of the lemma implies that any regeneration time for $\Y$ is also  a regeneration time for $\zb$ and $\zb'$.  To see this note that for any $N\in \N$,
\begin{align*}
\min_{m\ge N}Y_N>\max_{n<N}Y_n
\quad\Leftrightarrow \quad& Y_m-Y_n>0, && \forall n<N, m\ge N\\
\Leftrightarrow\quad& \sum_{i=n+1}^{m} (Y_{i}-Y_{i-1})>0, && \forall n<N, m\ge N \\
\Rightarrow\quad& \sum_{i=n+1}^{m} (Z_i-Z_{i-1})>0, && \forall n<N, m\ge N\\
\Rightarrow\quad&  Z_m-Z_n>0, && \forall n<N, m\ge N\\
\Leftrightarrow\quad& \min_{m\ge N}Z_m>\max_{n<N}Z_n.
\end{align*}

Since $\Y$ is a biased random walk, we have that \[\P(0 \in \mc{D}_Y)=\P(Y_n\ge 0,\, \forall n)=\frac{2p-1}{p}>0.\]
Moreover, $(\tau_{i+1}-\tau_i)_{i \in \N}$ are i.i.d.~random variables with finite mean \[\E[\tau_{2}-\tau_1]=
\overline{\E}[\tau_1]<\infty,\]
and $\E[\tau_1]<\infty$.  By (ii) we have that $((Z_{\tau_{i+1}}-Z_{\tau_i}, \tau_i))_{i \in \N}$ are i.i.d., and since $\zb$ is a nearest-neighbour walk we have that $\E[Z_{\tau_{1}}]\le \E[\tau_1]<\infty$, and similarly $\E[Z_{\tau_{2}}-Z_{\tau_1}]\le \E[\tau_2-\tau_1]<\infty$.

Thus, by standard regeneration arguments we get that 
\begin{align}
n^{-1}Z_n\ra v=\frac{\overline{\E}[Z_{\tau_1}]}{\overline{\E}[\tau_1]}, \qquad \textrm{and}\qquad n^{-1}Z'_n\ra v'=\frac{\overline{\E}[Z'_{\tau_1}]}{\overline{\E}[\tau_1]}, \qquad \textrm{almost surely},
\end{align}
with $v,v'>0$.  The fact that $v>v'$ is now immediate from (iii).\hfill\qed

\subsubsection{Proof of Lemma \ref{lem:tree_walks}}
Note that
\begin{align}
\overline{\E}\left[|X_{\tau_1}|-|X'_{\tau_1}|\right]&=\overline{\E}\left[(|X_{\tau_1}|-|X'_{\tau_1}|) \indic{|\mathfrak{B}|=1}\right]\\
&\qquad +\sum_{k=2}^{\infty}\overline{\E}\left[(|X_{\tau_1}|-|X'_{\tau_1}|) \indic{|\mathfrak{B}|=k}\right].
\end{align}

Since $|\X|$ and $|\X'|$ both increase by 1 whenever $\Y$ does, $|X|-|X'|$ can only change (by at most 2) when $\Y$ takes a backward step.  Thus on the event $\{|\mathfrak{B}|=k\}$ we have $|X_{\tau_1}|-|X'_{\tau_1}| \ge -2k$.  On the other hand, by assumption, 
$\overline{\P}\big(|\mathfrak{B}|=1, |X_{\delta}|-|X'_{\delta}|<0\big)=0$.  Therefore,
\begin{align}
\overline{\E}\left[|X_{\tau_1}|-|X'_{\tau_1}|\right]&\ge 
\overline{\P}\big(|\mathfrak{B}|=1, |X_{\tau_1}|-|X'_{\tau_1}|\ge1\big)\\
&\qquad -\sum_{k=2}^{\infty}2k\overline{\P}\left(|\mathfrak{B}|=k, |X_{\tau_1}|-|X'_{\tau_1}|\ne0\right).
\end{align}
and the result follows.\hfill \qed

\subsection{Application to the multiplicative case} \label{proofmult}

Here we detail the proof in the multiplicative case. Let us first indicate the key points of the argument.

As explained at the beginning of Section \ref{sec:mono}, the general idea is to compare three walks: a one-dimensional walk ${\bf Y}$ as well as two once-reinforced random walks $\Xb$ and $\Xe$ on $\T_d$ with respective reinforcement parameter $\beta$ and $\beta+\vep$. We will couple these three walks such that:
\begin{itemize}
\item ${\bf Y}$ has a strong bias towards the infinity, hence has regeneration times;
\item whenever ${\bf Y}$ steps forward, the three walks step forward;
\item every time ${\bf Y}$ steps back, the walk $\Xb$ and $\Xe$ increase their distance by at most $2$.
\end{itemize}
These points allow us to prove a result similar to Lemma \ref{almadelta}. Moreover, we will require that:
\begin{itemize}
\item on the event on which ${\bf Y}$ steps back only once before $\tau_1$, we a.s.~have that $|\xb_{\tau_1}|-|\xe_{\tau_1}|\ge0$;
\item the event on which ${\bf Y}$ steps back only once before $\tau_1$ and $|\xb_{\tau_1}|-|\xe_{\tau_1}|\ge1$, or any constant, happens with a probability roughly $\vep/\alpha d$, see Lemma \ref{mind1};
\item the event on which ${\bf Y}$ steps back $k$ times before $\tau_1$  happens with a probability roughly at most $\vep k(\alpha d)^{-k}$, see Lemma \ref{majdk}.
\end{itemize}
If these points are satisfied, then one can prove monotonicity properties as soon as $\alpha d$ is large enough, following the argument in Section \ref{sec:end}.\\
\subsubsection{The coupling}
\label{sec:coupling}

Let us now give a detailed construction of the coupling.\\

For any $k\in\{0,\dots,d\}$,  for any $\beta,\vep>0$, let
\begin{align}
p^{(\beta)}_{k}&=\frac{\alpha\1_{\{k<d\}}}{\alpha(d+k\beta) +1+\beta},\\
\overline{p}^{(\beta)}_{k}&=\frac{\alpha(1+\beta)\1_{\{k>0\}}}{\alpha(d+k\beta) +1+\beta},\\
q^{(\beta)}_{k}&=\frac{1+\beta}{\alpha(d+k\beta) +1+\beta}.\label{multiform}
\end{align}
The above correspond to the probability of stepping to an unvisited child, a visited child, and the parent respectively, when $k$ children of the current location have been visited previously (as long as the current location is not $\r^{-1}$).  Note that because of the multiplicative form of the reinforcement, when $k=d$ (all children have been previously  visited) the two relevant probabilities no longer depend on $\beta$.  Also note that 
\begin{equation}
(d-k)p^{(\beta)}_k+k\overline{p}^{(\beta)}_{k}+q^{(\beta)}_{k}=1.\label{sum_probs}
\end{equation}
We also define
\[
\Delta^{(p)}_k=p^{(\beta)}_k-p^{(\beta+\vep)}_k, \qquad \Delta^{(q)}_k=q^{(\beta+\vep)}_k-q^{(\beta)}_k, \qquad \overline{\Delta}_k= k\left(\overline{p}^{(\beta+\vep)}_{k}-\overline{p}^{(\beta)}_{k}\right), 
\]
which are all easily checked to be non-negative.  It is also easy to check that
\begin{equation}
\Delta_k^{(q)}=\frac{\vep \alpha (d-k)}{(\alpha(d+k(\beta+\vep)) +1+\beta+\vep)(\alpha(d+k\beta) +1+\beta)}\le \Delta_0^{(q)}.\label{Dbound1}
\end{equation}
Similarly, letting $\eta=1+\beta+\alpha d$,
\begin{equation}
\overline{\Delta}_k=k\alpha \Delta_k^{(q)}\le \frac{\vep  \alpha^2 k(d-k)}{(\eta+\vep)\eta}\le \frac{\vep  \alpha^2 d^2}{4(\eta+\vep)\eta},\label{Dbound2}
\end{equation}
where we have used the fact that $k(d-k)\le 4^{-1}d^2$.
From \eqref{sum_probs} we see that
\begin{align}
\label{sumdeltas}
-(d-k)\Delta^{(p)}_k+\overline{\Delta}_k+\Delta^{(q)}_k=0.
\end{align}

Recall the definition of $\mc{E}_\varnothing(n)$, which refers to the process $\Xb$.
For any $x\in\T_d$, $\beta\ge 0$,
let $C^{(\beta)}_n(x)$ be the number of children of $x$ visited by $\Xb$ up to time $n$, i.e.
\[
C^{(\beta)}_n(x)=\sum_{i=1}^d\indic{[x,x_i]\in \mc{E}_\varnothing(n)}.
\]
Let $K_n^{(\beta)}=C^{(\beta)}_n(\xb_n)$.
For $x\in\T_d$, and $i$ such that $1\le i\le C^{(\beta)}_n(x)$ we let $x^{\sss(\beta)}_{n,i}$ denote the $i$-th child of $x$ visited by the walk $(\xb_k)_{k\le n}$ (in order of visitation).   For $i$ such that $C^{(\beta)}_n(x)<i\le d$ we choose some fixed but arbitrary labelling $x^{\sss(\beta)}_{n,i}$.

{We now turn to the construction of the coupling.} Let $(U_i)_{i\ge1}$ be an i.i.d.~collection of uniform random variables on $[0,1]$. We define the coupling inductively, and include some figures to aid the understanding.  Figure \ref{fig:coupling} depicts the subintervals in which each of the scenarios (c1)-(c5) below apply.  In the figures below, red edges denote reinforced edges, the blue arrow indicates the move  of the walk $\Xe$, and the green arrow indicates the move of the walk $\Xb$.  Moves to the right are away from the root.  We set $x=\xb_n$ and $y=\xe_n$. {We first describe the case when $\Xb$ and $\Xe$ have visited the same number of visited children: this will be the main contribution, and the other case will become an error term in our computations.}

\begin{figure}
\begin{center}
\includegraphics[scale=1]{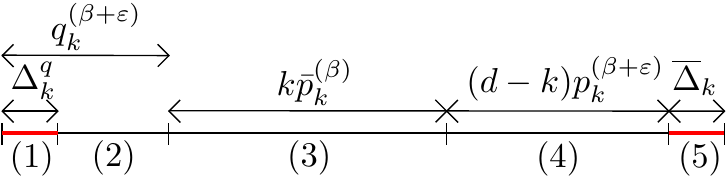}
\end{center}
\caption{A depiction of the coupling (c1)-(c5) of Section \ref{sec:coupling}}
\label{fig:coupling}
\end{figure}

Suppose that 
 $K_n^{(\beta+\vep)}=K_n^{(\beta)}=k$, where $k\in \{0,1,\dots, d\}$, and that $\r^{-1}\notin \{x,y\}$.
Then 
\begin{enumerate}
\item[(c1)] If $U_{n+1}\in( \frac{i-k-1}{d-k}\Delta^{(q)}_k, \frac{i-k}{d-k}\Delta^{(q)}_k]$ and if $k<d$, for $i\in\{k+1,...,d\}$, then $\xe_{n+1}=\parent{y}$, $\xb_{n+1}=x^{\sss(\beta)}_{n,i}$;\\
\[\]
\includegraphics[scale=.5]{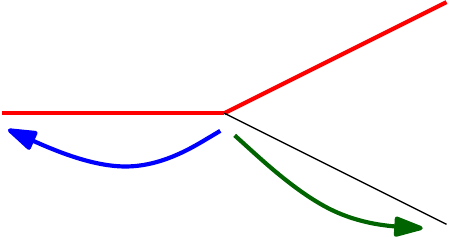}

\item[(c2)] If $U_{n+1}\in( \Delta^{(q)}_k, q^{(\beta+\vep)}_k]$ then $\xe_{n+1}=\parent{y}$ and $\xb_{n+1}=\parent{x}$;\\
\[\]
\includegraphics[scale=.5]{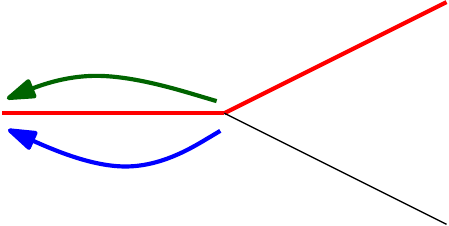}

\item[(c3)] If $U_{n+1}\in( q^{(\beta+\vep)}_k+(i-1)\overline{p}_k^{(\beta)},q^{(\beta+\vep)}_k+i\overline{p}_k^{(\beta)}]$, for $i\in\{1,\dots,k\}$, and if $k>0$ then $\xe_{n+1}=y^{\sss(\beta+\vep)}_{n,i}$ and $\xb_{n+1}=x^{\sss(\beta)}_{n,i}$;\\
\[\]
\includegraphics[scale=.5]{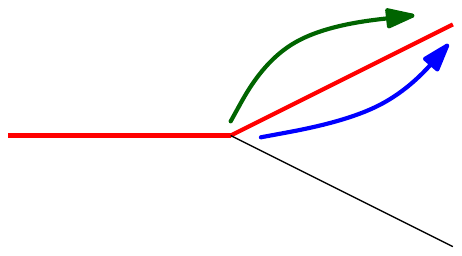}

\item[(c4)] If $U_{n+1}\in(q^{(\beta+\vep)}_k+k\overline{p}_k^{(\beta)}+(i-k-1)p_k^{(\beta+\vep)},q^{(\beta+\vep)}_k+k\overline{p}_k^{(\beta)}+(i-k)p_k^{(\beta+\vep)} ]$, for $i\in\{k+1,...,d\}$, and if $k<d$ then $\xe_{n+1}=y^{\sss(\beta+\vep)}_{n,i}$ and $\xb_{n+1}=x^{\sss(\beta)}_{n,i}$;\\
\[\]
\includegraphics[scale=.5]{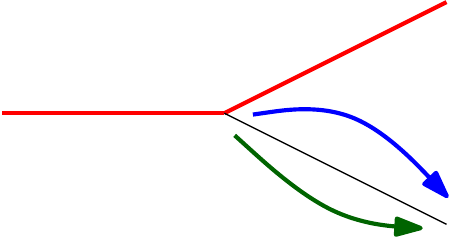}

\item[(c5)] If $U_{n+1}\in(q^{(\beta+\vep)}_k+k\overline{p}_k^{(\beta)}+(d-k)p_k^{(\beta+\vep)},1 ]=\left(1-\overline{\Delta}_k,1 \right]$ and if $0<k<d$ then
\begin{itemize}
\item[(a)] if $U_{n+1}\in(1-(1-\frac{i-1}{k})\overline{\Delta}_k,1-(1-\frac{i}{k})\overline{\Delta}_k ]$, for $i\in\{1,\dots,k\}$, then $\xe_{n+1}=y^{\sss(\beta+\vep)}_{n,i}$;
\item[(b)] if $U_{n+1}\in(1-(1-\frac{i-k-1}{d-k})\overline{\Delta}_k,1-(1-\frac{i-k}{d-k})\overline{\Delta}_k ]$, for $i\in\{k+1,\dots,d\}$, then $\xb_{n+1}=x^{\sss(\beta)}_{n,i}$.
\end{itemize}
\[\]
\includegraphics[scale=.5]{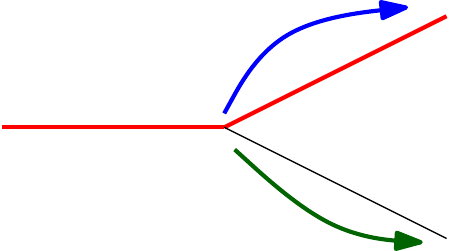}
\end{enumerate}
In the case where $k=d$ above neither condition (c4) nor (c5) can hold (as both require $k<d$), but the upper endpoint of the interval $(q^{(\beta+\vep)}_k+(i-1)\overline{p}_k^{(\beta)},q^{(\beta+\vep)}_k+i\overline{p}_k^{(\beta)}]$ is 1 when $i=d$.

From (c1) and (c2), we have that $\Xe$ steps to its parent when $U_{n+1}\in(0,q_k^{(\beta+\vep)}]$, which has probability $q_k^{(\beta+\vep)}$ as required.  Similarly from (c2) $\xb$ steps to its parent when $U_{n+1}\in( \Delta^{(q)}_k, q^{(\beta+\vep)}_k]$, which has probability $q_k^{(\beta)}$ as required.

From (c3), if $k=d$ then $\Xe$ steps to a child (they have all been previously visited) $y_{n,i}^{(\beta+\vep)}$ if $U_{n+1}\in ( q^{(\beta+\vep)}_k+(i-1)\overline{p}_k^{(\beta)},q^{(\beta+\vep)}_k+i\overline{p}_k^{(\beta)}]$, which has probability $\overline{p}_k^{(\beta)}=\frac{\alpha}{d\alpha+1}$ as required, and the same is true for $\Xb$.

From (c3) and (c5)(a), if $0<k<d$ then $\Xe$ steps to a previously visited child $y_{n,i}^{(\beta+\vep)}$ if 
\begin{align}
U_{n+1}\in  ( q^{(\beta+\vep)}_k+(i-1)\overline{p}_k^{(\beta)},q^{(\beta+\vep)}_k+i\overline{p}_k^{(\beta)}]\cup (1-(1-\frac{i-1}{k})\overline{\Delta}_k,1-(1-\frac{i}{k})\overline{\Delta}_k ].
\end{align}
This has probability $\overline{p}_k^{(\beta)}+\frac{1}{k}\overline{\Delta}_k=\overline{p}_k^{(\beta+\vep)}$ as required.  Similarly, from (c3), if $0<k<d$ then $\xb$ steps to a previously visited child $x_{n,i}^{(\beta+\vep)}$ if 
\begin{align}
U_{n+1}\in  ( q^{(\beta+\vep)}_k+(i-1)\overline{p}_k^{(\beta)},q^{(\beta+\vep)}_k+i\overline{p}_k^{(\beta)}]
\end{align}
which has probability $\overline{p}_k^{(\beta)}$ as required.

From (c1), (c4) and (c5)(b), if $0<k<d$ then $\xb$ steps to a previously unvisited child $x_{n,i}^{(\beta+\vep)}$ if 
\begin{align}
U_{n+1}\in &\Big( \frac{i-k-1}{d-k}\Delta^{(q)}_k, \frac{i-k}{d-k}\Delta^{(q)}_k\Big]\\
&\cup 
\Big(q^{(\beta+\vep)}_k+k\overline{p}_k^{(\beta)}+(i-k-1)p_k^{(\beta+\vep)},q^{(\beta+\vep)}_k+k\overline{p}_k^{(\beta)}+(i-k)p_k^{(\beta+\vep)} \Big]\\
&\cup \Big(1-(1-\frac{i-k-1}{d-k})\overline{\Delta}_k,1-(1-\frac{i-k}{d-k})\overline{\Delta}_k \Big],
\end{align}
which has probability $p_k^{(\beta+\vep)}+\frac{1}{d-k}(\Delta^{(q)}_k+\overline{\Delta}_k)$.  From \eqref{sumdeltas} this is equal to $p_k^{(\beta)}$, as required.  Similarly from (c4) $\Xe$ steps to any previously unvisited child $y_{n,i}^{(\beta+\vep)}$ with probability $p_k^{(\beta+\vep)}$ as required.

\bigskip

Otherwise either $K_n^{\beta}\ne K_n^{\beta+\vep}$ or $\r^{-1}\in \{x,y\}$.  If $y=\r^{-1}$, then $\xe_{n+1}=\r$ and similarly if $x=\r^{-1}$ then $\xb_{n+1}=\r$.  Otherwise,
\begin{enumerate}
\item[(c6)] Suppose that $K_n^{\beta+\vep}=k$.
\begin{enumerate}
\item[(a)] If $U_{n+1}\le q^{(\beta+\vep)}_{k}$ then $\xe_{n+1}=\parent{y}$;
\item[(b)] if $k>0$ and if \[U_{n+1}\in\left( q^{(\beta+\vep)}_{k} + (i-1)\overline{p}_{k}^{(\beta+\vep)},q^{(\beta+\vep)}_{k} + i\overline{p}_{k}^{(\beta+\vep)}\right]\] for $i\in\{1,\dots,k\}$ then $\xe_{n+1}=y^{\sss(\beta+\vep)}_{n,i}$;
\item[(c)] if $k<d$ and if 
\[U_{n+1}\in\left( q^{(\beta+\vep)}_{k} + k\overline{p}_{k}^{(\beta+\vep)}+(i-k-1){p}_{k}^{(\beta+\vep)},q^{(\beta+\vep)}_{k} + k\overline{p}_{k}^{(\beta+\vep)}+(i-k){p}_{k}^{(\beta+\vep)}\right]\]
for $i\in\{k+1,\dots,d\}$ then $\xe_{n+1}=y^{\sss(\beta+\vep)}_{n,i}$;
\end{enumerate}
\item[(c7)] Suppose that $K_n^{\beta}=k'$
\begin{enumerate}
\item[(a)] If $U_{n+1}\le q^{(\beta)}_{k'}$ then $\xb_{n+1}=\parent{x}$; 
\item[(b)] if $k'>0$  and if $U_{n+1}\in\left( q^{(\beta)}_{k'} + (i-1)\overline{p}_{k'}^{(\beta)},q^{(\beta)}_{k'} + i\overline{p}_{k'}^{(\beta)}\right] $ for $i\in\{1,\dots,k'\}$ then $\xb_{n+1}=x^{\sss(\beta)}_{n,i}$; 
\item[(c)] if $k'<d$ and if 
\[U_{n+1}\in\left( q^{(\beta)}_{k'} + k'\overline{p}_{k'}^{(\beta)}+(i-k'-1){p}_{k'}^{(\beta)},q^{(\beta)}_{k'} + k'\overline{p}_{k'}^{(\beta)}+(i-k'){p}_{k'}^{(\beta)}\right]\]
 for $i\in\{k'+1,\dots,d\}$ then $\xb_{n+1}=x^{\sss(\beta)}_{n,i}$.
\end{enumerate}
\end{enumerate}
From cases (a) above, the probability of moving to one's parent is $q^{(\beta+\vep)}_{k}$ for $\Xe$, and $q^{(\beta)}_{k'}$ for $\Xb$  as required.

From cases (b) above, the probability of moving to a particular previously visited child is $\overline{p}_{k}^{(\beta+\vep)}$ for $\Xe$ and $\overline{p}_{k'}^{(\beta)}$ for $\Xb$, as required.

From cases (c) above,  the probability of moving to a particular previously unvisited child is $p_{k}^{(\beta+\vep)}$ for $\Xe$ and $p_{k'}^{(\beta)}$ for $\Xb$ as required.

\begin{REM}
We have now verified that the above is a legitimate coupling of walks $\Xb$ and $\Xe$ with the correct marginal distributions.
\end{REM}

We now enrich the above coupling with a random walk $\Y$ on $\Z$, by setting for any $n\in\N$,
\begin{equation}\label{defY}
Y_n=\sum_{i=1}^n\left(\1_{\{U_i> q^{(\beta+\vep)}_0\}}-\1_{\{U_i\le q^{(\beta+\vep)}_0\}}\right).
\end{equation}
Let $\mc{F}_n=\sigma(U_k:k\le n)$, so that $\left(\mathcal{F}_n\right)_{n \in \N}$ is the  natural filtration generated by the sequence $(U_n)_{n \in \N}$ and note that all three walks are measurable with respect to this filtration. 
Also, define the probability that $0$ is a regeneration time for $Y$ as soon as $\alpha d>1+\beta+\vep$, that is
\[
p_\infty={\P}(\mc{D}_0)=\frac{\alpha d-(1+\beta+\vep)}{\alpha d}.
\]
\subsubsection{Checking the requirements of Lemma \ref{lem:super0}}
Here, we check that the three walks defined in the previous Section on the same probability space fulfill the requirements of Lemma \ref{lem:super0} $(i)$ and $(ii)$ and thus those of Lemma \ref{lem:tree_walks}.
\begin{LEM}
\label{lem:super}
For any $n$ such that $Y_{n+1}=Y_n+1$, $|\xb_{n+1}|=|\xb_{n}|+1$ and $|\xe_{n+1}|=|\xe_{n}|+1$.
Moreover, the regeneration times of $\Y$ are also regeneration time for $|\Xb|$ and $|\Xe|$, and they induce an i.i.d.~structure. In other words, $\Y$, $\Xb$ and $\Xe$ fulfill the requirements of Lemma \ref{lem:super0} $(i)$ and $(ii)$ and Lemma \ref{lem:tree_walks} can be applied.
\end{LEM}
\proof Firstly note that $\{Y_{n+1}-Y_n\ne 1\}=\{U_{n+1}\le q_0^{(\beta+\vep)}\}$.  As described above, each of $|\xe_n|$ and $|\xb_n|$ can  decrease on the next step only if $U_{n+1}\le q_k^{(\beta+\vep)}$.  Since $q_k^{(\beta+\vep)}$ is decreasing in $k$, $\{U_{n+1}\le q_k^{(\beta+\vep)}\}\subset\{U_{n+1}\le q_0^{(\beta+\vep)}\}$ and Lemma \ref{lem:super0} $(i)$ is satisfied, which implies that the regeneration times of $\Y$ are regeneration times for $\Xb$ and $\Xe$.\\
As $\Xb$ and $\Xe$ reach a new maximum at time $\tau_i$, never go back below $X_{\tau_i}^{\beta}$ and $X_{\tau_i}^{\beta+\vep}$, and because the steps of these walks depend only on the local environment, it is clear that the item $(ii)$ of Lemma \ref{lem:super0} holds as well.\hfill
\qed

\subsubsection{Decoupling events}

Define, for any $k\ge 2$,
\[
D_k=\left\{ |\xb_{\tau_1}|-|\xe_{\tau_1}|\ne 0\right\}\cap\left\{\left|\mathfrak{B}\right|=k\right\},
\]
By Lemma \ref{lem:tree_walks}, we easily obtain the following result.

\begin{LEM}\label{almadelta}
If it holds that
\begin{itemize}
\item[(i)] $\overline{\P}\big(|\mathfrak{B}|=1, |\xb_{\tau_1}|-|\xe_{\tau_1}|<0\big)=0$; and
\item[(ii)] $
\overline{\P}\big(|\mathfrak{B}|=1, |\xb_{\tau_1}|-|\xe_{\tau_1}|\ge 1\big)>\sum_{k=2}^{\infty}2k\overline{\P}(D_k),
$
\end{itemize}
then $v(\beta+\vep)<v(\beta)$.
\end{LEM}

We first prove $(i)$ and a lower bound on the left-hand side of $(ii)$. Set
\[
r=\frac{1+\beta+\vep}{\alpha d}.
\]
\begin{LEM}\label{mind1}
We have that $\overline{\P}\big(|\mathfrak{B}|=1, |\xb_{\tau_1}|-|\xe_{\tau_1}|<0\big)=0$ and
\begin{equation}\label{eq:boh}
\overline{\P}\big(|\mathfrak{B}|=1, |\xb_{\tau_1}|-|\xe_{\tau_1}|\ge 1\big)\ge \left(1-q_0^{(\beta+\vep)}\right)^3\Delta^{(q)}_0\ge\frac{\vep }{\alpha d(r+1)^5}.
\end{equation}
\end{LEM}
\begin{proof}
On $\{|\mathfrak{B}|=1\}$, the first and only time the walk $\mathbf{Y}$ takes a step back, then $\Xb$ and $\Xe$ are on a vertex whose children have not been   previously visited by the respective process. Therefore $\Xe$ necessarily takes a step back as well. After this time, the three walks keep on jumping forward until time $\tau_1$. This implies that no negative discrepancy can be created on $\{|\mathfrak{B}|=1\}$, proving the first part of the statement.\\
As for \eqref{eq:boh} note that
\[
\overline{\P}\big(|\mathfrak{B}|=1, |\xb_{\tau_1}|-|\xe_{\tau_1}|\ge 1\big)=\frac{1}{p_\infty}\P\left(|\mathfrak{B}|=1, |\xb_{\tau_1}|-|\xe_{\tau_1}|\ge 1, \mc{D}_0\right).
\]

Consider the event on which the increments of $\Y$ are (in order) $+1,-1,+1,+1$, with $\tau_1=4$, and that the second step of $\Xe$ is towards the root while the second step of $\Xb$ is away from the root.  One this event, $\left\{|\mathfrak{B}|=1, |\xb_{\tau_1}|-|\xe_{\tau_1}|\ge 1, \mc{D}_0\right\}$ holds. Hence,
\[
\overline{\P}(|\mathfrak{B}|=1, |\xb_{\tau_1}|-|\xe_{\tau_1}|\ge 1)\ge\frac{1}{p_\infty}\times (1-q_0^{(\beta+\vep)})\times \Delta^{(q)}_0\times\left(1-q_0^{(\beta+\vep)}\right)^2\times p_\infty,
\]
and a simple computation gives 
\[\Delta^{(q)}_0=\frac{\vep \alpha d}{(\eta+\vep)\eta}\]
as required. \hfill
\end{proof}

Next, we provide an upper-bound on the right-hand side of $(ii)$ in Lemma \ref{almadelta}.

\begin{LEM}\label{majdk}
Recall that $r=({1+\beta+\vep})/{\alpha d}$. For $k\ge 2$,
\[
\overline{\P}\left(D_k\right)\le   \frac{1}{3} \vep k p_\infty^{-1} \left(\frac{27}{4} q_0^{\sss(\beta+\vep)}\right)^{k}=  \frac{\vep}{3(1-r)}  k \left(\frac{27}{4}\times\frac{r}{r+1}\right)^{k}
\]
\end{LEM}

\begin{proof}
First, recall that $\delta$ is the time of decoupling and note that necessarily $Y_{n+1}-Y_n=1$ for each $n\in\{\tau_1-2,\tau_1-1\}$ otherwise $Y_{\tau_1}\le Y_{\tau_1-2}$ which contradicts the definition of regeneration times. Hence, on $\mathcal{D}_0\cap D_k$,  we almost surely have that $|\{n\in\{1,\dots,\tau_1-3\}:Y_{n+1}-Y_n=-1\}|= k$.  Using Lemma \ref{majtau} this implies that  $|\{n\in\{1,\dots,3k-2\}:Y_{n+1}-Y_n=-1\}|\ge k$, almost surely on $\mathcal{D}_0\cap D_k$.  Therefore
\begin{align}\nonumber
\overline{\P}\left(D_k\right)=&p_\infty^{-1}\P(\mathcal{D}_0,D_k,\tau_1\le 3k+1)\\
\le& p_\infty^{-1}{\P}\left(2\le\delta\le3k+1, |\{1\le j \le3k-2:Y_{j+1}-Y_j=-1\}|\ge k\right),\label{trousers1}
\end{align}
where we have used the fact that on $\mc{D}_0$ we must have $Y_1=1$ and $\delta\ge 2$.
\item 

Let $A^1_{n,c}=\{U_n\le\Delta_c^{(q)}\}$ and $A^2_{n,c}=\{U_n>1-\overline{\Delta}_c\}$, as well as $A^1_n=\{U_n\le\Delta_0^{(q)}\}$ and $A^2_n=\{U_n>1-\frac{\vep  \alpha^2 d^2}{4(\eta+\vep)\eta}\}$.  Then by \eqref{Dbound1} and \eqref{Dbound2} we have that $A^1_{n,c}\subset A^1_{n}$ and $A^2_{n,c}\subset A^2_{n}$. On the event $\{\delta=n\}$ we have that $K_{n-1}^{(\beta)}=K_{n-1}^{(\beta+\vep)}$.  Hence from (c1) and (c5), 
\begin{eqnarray}
\left\{\delta=n\right\}&\subset&\bigcup_{c=1}^d \left(\left(A^1_{n,c}\cup A^2_{n,c}\right) \cap \left\{K_{n-1}^{(\beta)}=c=K_{n-1}^{(\beta+\vep)}\right\}\right)\nn\\
&\subset & \bigcup_{c=1}^d \left(\left(A^1_{n}\cup A^2_{n}\right) \cap \left\{K_{n-1}^{(\beta)}=c=K_{n-1}^{(\beta+\vep)}\right\}\right)\nn\\
&=&\left(A^1_{n}\cup A^2_{n}\right)\cap \left(\bigcup_{c=1}^d\left\{K_{n-1}^{(\beta)}=c=K_{n-1}^{(\beta+\vep)}\right\}\right)\nn\\
&\subset& A^1_{n}\cup A^2_{n}.\label{banana1}
\end{eqnarray}
Note that $A^1_n$ and $A^2_n$ are measurable w.r.t.~$U_{n}$ only and that $A^1_n\cap A^2_n=\emptyset$ as soon as $\vep$ is small enough. Therefore, ${\bf Y}$ takes a step back on $A^1_n$, but takes a step forward on  $A^2_n$.
Using \eqref{banana1} in \eqref{trousers1} we obtain
\begin{equation}\label{eq:boh1}
\begin{aligned}
\overline{\P}\left(D_k\right)\le& p_\infty^{-1}
\sum_{n=2}^{3k+1}{\P}\left(A^1_{n},  |\{1\le j \le3k-2:Y_{j+1}-Y_j=-1\}\setminus\{n\}|\ge k-1\right)\\
& + p_\infty^{-1}
\sum_{n=2}^{3k+1}{\P}\left(A^2_{n},  |\{1\le j \le3k-2:Y_{j+1}-Y_j=-1\}\setminus\{n\}|\ge k\right).
\end{aligned}
\end{equation}
Written in terms of the variables $U_j$, and using independence and the fact that $\P(A^1_n)~~=~~\P(A^1_1)$, the first term in the right hand-side of \eqref{eq:boh1}  is equal to
\begin{eqnarray}
&&p_\infty^{-1}\sum_{n=2}^{3k+1}{\P}\left(A^1_{n}, |\{j\in\{1,\dots,3k-2\}\setminus\{n\}  :U_j\le q^{(\beta+\vep)}_0\}|\ge k-1\right)\\ \nonumber
&=&p_\infty^{-1}\sum_{n=2}^{3k+1}\P(A^1_n){\P}\left(|\{j\in\{1,\dots,3k-2\}\setminus\{n\}  :U_j\le q^{(\beta+\vep)}_0\}|\ge k-1\right)\\
&=& p_\infty^{-1}\P(A^1_1)\sum_{n=2}^{3k+1}\P\left(|\{j\in\{1,\dots,3k-2\} \setminus\{n\}  :U_j\le q^{(\beta+\vep)}_0\}|\ge k-1\right)\\
&\le&3k p_\infty^{-1}\P(A^1_1)\P\left(|\{j\in\{1,\dots,3k-3\}  :U_j\le q^{(\beta+\vep)}_0\}|\ge k-1\right)\\
&=&3k p_\infty^{-1}\P(A^1_1)\P\left(\exists B\subset \{1,\dots,3(k-1)\}: |B|=k-1,U_j\le q^{(\beta+\vep)}_0 \textrm{ for all }j \in B\right)\nn \\
&\le &3k p_\infty^{-1}\P(A^1_1){3(k-1) \choose k-1}(q_0^{\sss(\beta+\vep)})^{k-1}. \label{trousers2}
\end{eqnarray}
Using the inequality $e^{11/12}\sqrt{n}(n/e)^n<n!<e\sqrt{n}(n/e)^n$, which holds for any $n\ge2$, see \cite{refcomb}, for $k\ge 3$ we have
\[{3(k-1) \choose k-1}\le e^{-10/12} \left(\frac{27}{4}\right)^{k-1}.\]
Together with the fact that $e^{10/12}>2$, we have that \eqref{trousers2} is at most
\begin{equation}
\frac{3}{2}k p_\infty^{-1}\P(A_1) \left(\frac{27}{4} q_0^{\sss(\beta+\vep)}\right)^{k-1}.
\end{equation}
Now $ \P(A_1^1)=\frac{\vep \alpha d}{(\eta+\vep)\eta}\le\vep  q_0^{\sss(\beta+\vep)}$, so this is at most
\begin{equation}\label{raelsan}
\frac{1}{4}k p_\infty^{-1}\vep \left(\frac{27}{4} q_0^{\sss(\beta+\vep)}\right)^{k}.
\end{equation}
Similarly, we have that
\begin{eqnarray*}
&&p_\infty^{-1}\sum_{n=2}^{3k+1}{\P}\left(A^2_{n}, |\{j\in\{1,\dots,3k-2\}\setminus\{n\}  :U_j\le q^{(\beta+\vep)}_0\}|\ge k\right)\\ \nonumber
&=&p_\infty^{-1}\sum_{n=2}^{3k+1}\P(A^2_n){\P}\left(|\{j\in\{1,\dots,3k-2\}\setminus\{n\}  :U_j\le q^{(\beta+\vep)}_0\}|\ge k\right)\\
&=&3k p_\infty^{-1}\P(A^1_1)\P\left(\exists B\subset \{1,\dots,3(k-1)\}: |B|=k,U_j\le q^{(\beta+\vep)}_0 \textrm{ for all }j \in B\right)\nn \\
&\le &3k p_\infty^{-1}\P(A^1_1){3(k-1) \choose k}(q_0^{\sss(\beta+\vep)})^{k}\nn \\
&\le &6k p_\infty^{-1}\P(A^1_1){3(k-1) \choose k-1}(q_0^{\sss(\beta+\vep)})^{k}  \le 3k p_\infty^{-1}\frac{\vep  \alpha^2 d^2}{4(\eta+\vep)\eta}\frac{4}{27}\left(\frac{27}{4} q_0^{\sss(\beta+\vep)}\right)^{k}\nn \\
&\le &  \vep k p_\infty^{-1}\frac{1 }{18}\left(\frac{27}{4} q_0^{\sss(\beta+\vep)}\right)^{k,}
\end{eqnarray*}
where we used  the  definition of $\eta=1+\beta+\alpha d$.
Together with \eqref{raelsan} and \eqref{eq:boh1} this provides the conclusion.\hfill
\end{proof}

%

\subsubsection{Conclusion and proof of Theorem \ref{thm:mono} in the multiplicative case} \label{sec:end}

\begin{proof}
Using Lemma \ref{mind1}, Lemma \ref{majdk} and the fact that $\sum_{k\ge2}k^2x^k\le 4x^2/(1-x)^3$ for any $x<1$, a straightforward computation yields
\begin{align*}
\frac{1}{\overline{\P}\big(|\mathfrak{B}|=1, |X'_{\tau_1}|-|X_{\tau_1}|\ge 1\big)}\sum_{k=2}^{\infty}2k\overline{\P}(D_k)&\le  \alpha d r^2\times \frac{3^5(1+r)^6}{2(1-r)(1-\frac{23}{4}r)^3}\\
\le & (1+\beta+\vep) r\times \frac{3^5(1+r)^6}{2(1-r)(1-\frac{23}{4}r)^3}.
\end{align*} 
One can check that the function defined by
\[
f(r)=r\times \frac{3^5(1+r)^6}{2(1-r)(1-\frac{23}{4}r)^3},
\]
has positive derivative and that $f(1/150)<1$. Recalling that $r=({1+\beta+\vep})/{\alpha d}$, this implies the conclusion, by Lemma \ref{almadelta}.\hfill
\end{proof}

\subsubsection{Improving the constants}\label{sectimprove}

One can improve the threshold for which we prove monotonicity. First,  when ${\bf Y}$ steps back only once and the walks decouple, then they create a discrepancy of $2$, as they have to decouple while having no reinforced children. In other words, the discrepancy created as Lemma \ref{mind1} is in fact $+2$.
Next, note that no negative discrepancy can be created when the walk ${\bf Y}$ takes only two steps back before time $\tau_1$. Indeed, the only possibilities on $D_2$ are that:
\begin{itemize}
\item the walks decouple at the first step back of ${\bf Y}$ and thus $|\xb_\delta|-|\xe_\delta|=2$. Therefore, as ${\bf Y}$ will take only one more step back $ |\xb_{\tau_1}|-|\xe_{\tau_1}|\ge0$.
\item the walks do not decouple at the first step back of ${\bf Y}$, hence they both step back and have a reinforced child. If ${\bf Y}$ steps back again right after, then, if the walks decouple at this time then $|\xb_\delta|-|\xe_\delta|\ge0$, and in any case they  only step forward up to time $\tau_1$, therefore $|\xb_{\tau_1}|-|\xe_{\tau_1}|\ge 0$. If ${\bf Y}$ does not step back right away, then, when it does, $\Xb$ and $\Xe$ will have no reinforced child, hence $ |\xb_{\tau_1}|-|\xe_{\tau_1}|=2$.
\item if the walks decouple strictly after the second back step of ${\bf Y}$, then $ |\xb_{\tau_1}|-|\xe_{\tau_1}|=0$.
\end{itemize}
 A consequence of this is that the item $(ii)$ in Lemma \ref{almadelta} can be replaced by 
 \[
\overline{\P}\big(|\mathfrak{B}|=1, |\xb_{\tau_1}|-|\xe_{\tau_1}|\ge 1\big)>\sum_{k=3}^{\infty}k\overline{\P}(D_k).
\]
In turn, mimicking the proof in Section \ref{sec:end}, this implies that we are looking for the greatest $r$ such that
\[
f(r)=r^2\times \frac{3^7(1+r)^6}{2^4(1-r)(1-\frac{23}{4}r)^3}<1.
\]
If $r\le 1/22$, then $f(r)<1$. This implies that monotonicity occurs for some non-empty interval of reinforcement parameter, as soon as $\alpha d\ge22$.

\subsection{Application to the additive case}\label{proofadd}

For the additive case, we can follow exactly the same blueprint as in the multiplicative case. The biggest difference is the way to define the coupling, which is a little bit more cumbersome. Indeed, the transition probabilities of $\Xb$ and $\Xe$ on a given vertex do not behave monotonically. More precisely, here is what happens:
\begin{itemize}
\item As long as the walks have the same local environment, $\Xe$ is always more likely to walk on its trace compared to $\Xb$;
\item if $\alpha<1$ and if, at a given time, $\Xb$ and $\Xe$ are still coupled and the number of reinforced children is greater than $\alpha d$, then $\Xb$ is more likely to step back compared to $\Xe$;
\item $\Xe$ is more likely to walk on a reinforced child than $\Xb$, except if $\alpha>1$ and if all children are reinforced (recall that in the additive case, we do not recover the initial law once all the children are reinforced).
\end{itemize}
The walk $Y$ would remain the same, just as in \eqref{defY}, taking a step back with probability $\frac{1+\beta+\vep}{1+\beta+\vep+\alpha d}$. 

We can then define a coupling of ${\bf Y}$, $\Xb$ and $\Xe$ in the additive case such that:
\begin{itemize}
\item Whenever ${\bf Y}$ steps forward, then $\Xb$ and $\Xe$ both step forward. Hence, if $\Xb$ or $\Xe$ steps back then $\Y$ steps back.
\item The probability that $\Xb$ and $\Xe$ decouple is proportional to $\vep$.
\item Whenever $\Xb$ walks on its trace, $\Xe$ walks on its trace. 
\item $\overline{\P}\big(|\mathfrak{B}|=1, |\xb_{\tau_1}|-|\xe_{\tau_1}|<0\big)=0$.
\end{itemize}

\subsection{Generalization to Galton-Watson trees}
It is legitimate to try to obtain a monotonicity result on Galton-Watson trees without leaves. Again, the only fact that we need in order to apply the method is a nice way to couple two once-reinforced random walks $\Xb$ and $\Xe$ to the same highly biased random walk on $\Z$. One way to do this is to apply the same procedure as in \cite{BAFS}: we couple the walks $\Xb$ and $\Xe$ in a similar manner as previously described, except that we define them on two different Galton-Watson trees that we sample sequentially as the walks explore them. Let $(Z_n)_{n \in \Z_+}$ be an i.i.d.~sequence of offspring random variables. If $\Xb$ or $\Xe$ discovers a new site at time $n$, then this site is given $Z_n$ children (note that $\Xb$ and $\Xe$ do not live on the same tree). Hence, every time the walks discover a new site at the same time, they are given the same offspring. As previously, the two walks are jointly coupled to a one-dimensional biased random walk jumping to the right with probability $(1+\beta+\vep)/(\alpha d+1+\beta+\vep)$, where $d=\min\{d\ge1:P(Z_0=d)>0\}$ is the minimal degree of the tree.\\
Such a coupling would allow us to conclude monotonicity in a non-empty interval for $\beta$ as soon as $\alpha d$ is large enough.

\section*{Acknowledgements}
MH's research was supported by a Marsden grant, administered by the Royal Society of NZ, and by an ARC Future Fellowship, FT160100166. DK is grateful to the University of Auckland, to Monash University for their hospitality and to the Ecole Polytechnique F\'ed\'erale de Lausanne to which he was affiliated to at the time this work was initiated. A.C. is grateful to New York University in Shanghai for its hospitality, and he was supported by ARC grants DP140100559 and DP180100613.

\end{document}